\documentclass[12pt]{article}
\usepackage{JosPackage}
\allowdisplaybreaks

\title{One-periodic category of twisted complexes, Cohen-Macaulay modules and differential modules}
\author{Joseph Winspeare}
\date{ }

\begin{document}
	\maketitle
	
	\section*{Abstract}
	
	In this paper, we give an $A_\infty$-enhancement of the triangulated orbit category $(\Tcat/[1])_{tr}$ when $\Tcat$ is an algebraic triangulated category. We then use this enhancement to give a description of $(\per(A)/[1])_{tr}$ in terms of differential modules and Cohen-Macaulay modules when $A$ is a finite dimensional $\Z$-graded algebra of finite global dimension.
	
	\tableofcontents
	
	\section*{Introduction}
	
	Given an algebraic triangulated category $\Tcat$ and an autoequivalence $F: \Tcat \to \Tcat$, one can form the orbit category $\Tcat/F$ whose objects are those of $\Tcat$ and morphism spaces are
	\[\Hom_{\Tcat/F}(X,Y):= \bigoplus_{\ell \in \Z} \Hom_\Tcat(X,F^\ell (Y)).\]
	This defines an additive category that is equipped with a canonical additive functor $\Tcat \to \Tcat/F$. In Keller's foundational paper \cite{KelOrb}, the author gives a sufficient condition on $\Tcat$ and $F$ to equip $\Tcat/F$ with a triangulated structure such that the canonical functor $\Tcat \to \Tcat/F$ is a functor of triangulated categories. To do so, he uses a DG-enhancement of $\Tcat$ to construct the triangulated orbit category $(\Tcat/F)_{tr}$, defined as the "smallest triangulated category containing $\Tcat/F$". This construction has many applications, most notably the categorification of cluster algebras via the generalized cluster category (see \cite{KelOrb,Ami_thesis,KelCY,Guo,KelDef}). Although the construction of the triangulated orbit category is technical, Keller gives an explicit alternative description in the case where $\Tcat = \Dcat^b(\mod \, A)$ with $A$ a finite-dimensional algebra of finite global dimension.
	
	In this paper, we construct and study an $A_\infty$-enhancement of the triangulated category $(\Tcat/[1])_{tr}$ where $[1]$ is the suspension functor of the triangulated structure on $\Tcat$. By assumption, $\Tcat$ has an $A_\infty$-enhancement of the form $\Tw(\Acat)$ and the autoequivalence $[1]$ lifts to the shift functor of $\Tw(\Acat)$. We call the $A_\infty$-enhancement of $(\Tcat/[1])_{tr}$ thus constructed the one-periodic $A_\infty$-category of twisted complexes over $\Acat$.
	
	We then focus on the case where $\Acat$ is given by a graded bound quiver $(Q,I)$ of finite global dimension. In this case, $\Tcat$ is the perfect derived category of $A:= kQ/I$ where $A$ is considered as a DG-algebra with zero differential. We obtain a description in terms of differential modules over $A$ and in terms of maximal Cohen-Macaulay modules over an extension $A^\ltimes$ of $A$. When $A$ is trivially graded, we recover results of \cite{Stai} and \cite{Win}. In general, this description requires us to extend the classical definition of differential modules to graded algebras (Definition~\ref{difmod}). 
	
	In the first section, we give a step by step construction of an $A_\infty$-enhancement of $(\per(\Acat)/[1])_{tr}$ where $\Acat$ is a DG-category. There are two ways of constructing the $A_\infty$-enhancement of $\per(\Acat)/[1]$: the DG-orbit category $\Tw(\Acat)/[1]$ as defined by Keller (see \cite{KelOrb}) and the skew $A_\infty$-category $\Tw(\Acat) *_{[1]} \Z$ as defined by Opper and Zvonareva (see \cite{OZ}). We show that these two constructions coincide (Proposition~\ref{sameorbit}). We then consider $\Tw(\Tw(\Acat) *_{[1]} \Z)$ as the $A_\infty$-enhancement of $(\per(A)/[1])_{tr}$ by virtue of \cite[Section~7.4]{Lef}. These enhancements are summarized in the following table.
	
	\begin{center}
		\begin{tabular}{|c|c|}
			\hline 
			Category & $A_\infty$-enhancement \\
			\hline
			$\per(\Acat)$ & $\Tw(\Acat)$ \\
			$\per(\Acat)/[1]$ & $\Tw(\Acat) *_{[1]} \Z$  \\
			$(\per(\Acat)/[1])_{tr}$ & $\Tw(\Tw(\Acat) *_{[1]} \Z)$ \\
			\hline
		\end{tabular}
	\end{center}
	
	In the second section of this paper, we study the $A_\infty$-category $\Tw(\Tw(\Acat) *_{[1]} \Z)$ when $\Acat$ is any $A_\infty$-category. We first show that it is quasi-equivalent to $\Tw(\Z \Acat * \Z)$ (Theorem~\ref{equione}). We then show that it is also quasi-equivalent to $\Tw(k[t,t^{-1}] \otimes_k \Acat)$ (Theorem~\ref{equitwo}). We finish this section by describing a subcategory $\Scat$ of $\Tw(k[t,t^{-1}] \otimes_k \Acat)$ such that the inclusion is a quasi-equivalence (Theorem~\ref{retract}). 
	
	In the third section, we restrict our study to the case where $\Acat$ is given by a bounded $\Z$-graded quiver with relations $(Q,I)$ such that the $k$-algebra $A:=kQ/I$ has finite global dimension. Extending the ungraded case considered in \cite{Win}, we consider the graded algebra $A \otimes_k k[\varepsilon]/\langle \varepsilon^2 \rangle$ where $|\varepsilon| = 1$ as an ungraded algebra denoted $A^\ltimes$. We then construct an equivalence of categories $\Gamma: \mod \, A^\ltimes \overset{\sim}{\to} \Dif \, A$ where $\Dif \, A$ is the category of finitely generated differential $A$-modules. Considering the Frobenius exact subcategory $\CM(A^\ltimes)$ of $\mod \, A^\ltimes$ consisting of maximal Cohen-Macaulay $A^\ltimes$-modules and its image $\Pcat(A)$ in $\Dif(A)$, we obtain an equivalence of triangulated categories
	
	\[(\per(A)/[1])_{tr} \simeq \underline{\Pcat(A)} \simeq \barCM(A^\ltimes)\]
	
	\noindent which is the main result of this paper (Corollary~\ref{main}).

	\subsection*{Conventions}
	
	In this paper, $k$ is a field, all categories are $k$-linear and all algebras are finite-diemsional $k$-algebras.
	
	\section{$A_\infty$ orbit categories}
	
	In this section, we recall definitions and construtions concerning the $A_\infty$-category of twisted complexes. We use conventions and notations of \cite{Sei}.
	
	\subsection{Basic definitions and notations}
	
	All $A_\infty$-categories are assumed to be strictly unital and $A_\infty$-functors to be strict and unital. In this paper, we use definitions and notations as given by Seidel in \cite{Sei}. If $a$ is a homogeneous morphism of an $A_\infty$-category $\Acat$, we write $|a|$ its degree and $\| a \|:= |a| - 1$ its reduced degree. 
	
	\begin{remark}
		\label{Ainftystr}
		If $\mu^d_\Acat=0$ for $d \geq 3$ then $\Acat$ can be equipped with a canonical DG-category structure (in the sense of \cite{KelDG}). Let $f:Y \to Z$ and $g:X \to Y$ be homogeneous morphisms, the differential and composition are defined as follows.
 		\[\partial(f) = (-1)^{|f|} \mu^1_\Acat(f)\]
 		\[f \circ g = (-1)^{|g|} \mu^2_\Acat(f,g)\]
 		With this construction, we can also equip any $k$-linear category or DG-category with an $A_\infty$-structure such that $\mu^d = 0$ for $d \geq 3$.
	\end{remark}
	
	Throughout this paper we will use the constructions $Z^0(\Acat)$ and $H^0(\Acat)$. For definitions, see \cite[Section 1a, Section 3k]{Sei}. Recall that these constructions are functorial, if $F: \Acat \to \Bcat$ is an $A_\infty$-functor, it induces functors $Z^0(F): Z^0(\Acat) \to Z^0(\Bcat)$ and $H^0(F): H^0(\Acat) \to H^0(\Bcat)$.
	
	The constructions of $\Z \Acat$ and $\Tw(\Acat)$ are now recalled in order to fix notations.

	\begin{definition}
		Let $\Acat$ be an $A_\infty$-category. The $A_\infty$-category $\Z \Acat$ is defined as follows.
		\begin{itemize}
			\item Objects of $\Z \Acat$ are formal finite direct sums $\displaystyle\bigoplus_{i \in I} X_i[p_i]$ where for all $i \in I$, $p_i$ is an integer.
			\item If $X$, $Y$ are objects of $\Acat$ and $p$, $q$ are integers then 
			\[\Hom_{\Z \Acat}(X[p], Y[q]):= \Hom_{\textup{gr} \, k}(k[p],k[q]) \otimes_k \Hom_\Acat(X,Y)\]
			where $\textup{gr} \, k$ is the $\Z$-graded category of $\Z$-graded $k$-modules.
			The degree $i$ component $\Hom^i(X[p],Y[q])$ is generated by morphisms of the form $s^p_q \otimes a$ where $s^p_q$ is the identity map from $k[p]$ to $k[q]$ (it has degree $p-q$). The morphism $s^p_q \otimes a$ has degree $p-q + |a|$.
			\item The identity map of $X[p]$ is $s^p_p \otimes \id_X$.
			\item Consider $\mu^d_{\Z \Acat}$ defined by 
			\[\mu_{\Z \Acat}^d(s^{p_{d-1}}_{p_d} \otimes a_d, \dots, s^{p_1}_{p_2} \otimes a_1):= (-1)^{\displaystyle\sum_{i < j} (p_{i-1} -p_i).\|a_j\|} s^{p_1}_{p_d} \otimes \mu_\Acat^d(a_d,\dots, a_1)\]
			on composable homogeneous morphisms. This extends to an $A_\infty$-structure on $\Z \Acat$ by linearity.
		\end{itemize}
	\end{definition}
	
	\begin{definition}
		Let $\Acat$ be an $A_\infty$-category. The $A_\infty$-category $\Tw(\Acat)$ is defined as follows.
		\begin{itemize}
			\item Objects of $\Tw(\Acat)$ are pairs $(X,\delta_X)$ where $X$ is an object of $\Z \Acat$ and $\delta_X: X \to X$ is a degree 1 endomorphism of $X$ in $\Z \Acat$ that is strictly lower triangular (for a detailled definition of strictly lower triangular, see \cite[Section 3l]{Sei}) and such that
			\begin{equation} 
			\sum_{d>0} \mu_{\Z \Acat}^d(\delta_X,\dots, \delta_X) = 0
			\end{equation}
			This equation is called the generalized Maurer-Cartan equation. We call objects of $\Tw(\Acat)$ twisted complexes over $\Acat$.
			\item If $(X,\delta_X)$ and $(Y,\delta_Y)$ are twisted complexes over $\Acat$, set
			\[\Hom_{\Tw(\Acat)}((X,\delta_X),(Y,\delta_Y)):= \Hom_{\Z \Acat}(X,Y)\]
			\item Consider the structure maps $\mu^d_{\Tw(\Acat)}$ given by
			\[\forall d >0, \;\mu_{\Tw(\Acat)}^d(f_d,\dots,f_1):= \sum_{i_0,\dots,i_d \geq 0} \mu_{\Z \Acat}^{d+i_0 + \dots + i_d}(\delta_{X_d}^{\otimes i_d},f_d,\delta_{X_{d-1}}^{\otimes i_{d-1}},a_{d-1},\dots,a_1,\delta_{X_0}^{\otimes i_0})\]
			on composable homogeneous morphisms. This defines an $A_\infty$-structure on $\Tw(\Acat)$ by linearity.
		\end{itemize}
	\end{definition}
	
	\begin{definition}
		A strict $A_\infty$-functor $F: \Acat \to \Bcat$ is said to be cohomologically fully faithful if the induced functor $H^0(F): H^0(\Acat) \to H^0(\Bcat)$ is fully faithful.
	\end{definition}
	
	\begin{remark}
		Let $\Acat$ be an $A_\infty$-category. There are cohomologically fully faithful $A_\infty$-functors $\Acat \hookrightarrow \Z \Acat \hookrightarrow \Tw(\Acat)$ defined as follows
		\begin{center}
			\begin{tikzpicture}
				\node (1) at (0,0) {$\Acat$};
				\node (2) at (5,0) {$ \Z\Acat$};
				\node (3) at (12,0) {$\Tw(\Acat)$};
				\node (4) at (0,-1) {$X$};
				\node (5) at (5,-1) {$X[0]$};
				\node (6) at (12,-1) {$(X[0],0)$};
				\node (7) at (0,-2) {$a:X\to Y$};
				\node (8) at (5,-2) {$s^0_0 \otimes a: X[0] \to Y[0]$};
				\node (9) at (12,-2) {$s^0_0 \otimes a: (X[0],0) \to (Y[0],0)$};
				
				\draw[right hook-stealth] (1) to (2);
				\draw[right hook-stealth] (2) to (3);
				\draw[|-stealth] (4) to (5);
				\draw[|-stealth] (5) to (6);
				\draw[|-stealth] (7) to (8);
				\draw[|-stealth] (8) to (9);
			\end{tikzpicture}
		\end{center}
		Furthermore, the construction of $\Tw(\Acat)$ is functorial, a strict $A_\infty$-functor $F: \Acat \to \Bcat$ induces a strict $A_\infty$-functor $\Tw(F): \Tw(\Acat) \to \Tw(\Bcat)$.
	\end{remark}
	
	\begin{definition}
		We say that an $A_\infty$-category $\Acat$ is triangulated if $H^0(\Acat)$ has a structure of triangulated category with respect to exact triangles is $\Acat$ (in the sense of \cite[Section 3h]{Sei}).
	\end{definition}
	
	\begin{proposition}{\cite[Lemma 3.28]{Sei}}
	Let $\Acat$ be an $A_\infty$-category. Then $\Tw(\Acat)$ is triangulated.
	\end{proposition}
	
	\begin{definition}
		\label{deftria}
		Let $\Acat$ be a triangulated $A_\infty$-category. A full subcategory $\Gcat$ of $\Acat$ is a generating subcategory if the inclusion $i: \Gcat \to \Acat$ induces a quasi-equivalence $\Tw(i): \Tw(\Gcat) \to \Tw(\Acat)$. In particular, $\Acat$ is a generating subcategory of $\Tw(\Acat)$.
	\end{definition}
	
	\noindent We now recall a useful result on triangulated $A_\infty$-categories.
	
	\begin{lemma}{\cite[Lemma 3.34]{Sei}}
	\label{Seidel}
	Let $F: \Acat \to \Bcat$ be a cohomologically full and faithfull strict $A_\infty$-functor. Assume that $\Bcat$ is triangulated and that the image of $F$ forms a generating subcategory of $\Bcat$. Then $F$ induces a quasi-equivalence $\tilde{F}: \Tw(\Acat) \to \Bcat$ that restricts to $F$ on $\Acat$.
	\end{lemma}
	
	An interpretation of this result concerns triangulated hulls. Consider an inclusion of $A_\infty$-categories $\Acat \subset \Bcat$ with $\Bcat$ triangulated, then $\Tw(\Acat)$ is quasi-equivalent to the smallest triangulated subcategory of $\Bcat$ containing $\Acat$.
	
	\subsection{Orbit categories}
	
	In this section, we will detail two constructions of orbit categories: the orbit of an $A_\infty$-category under a strict action of a group $G$ and the orbit category of a DG-category under the action of a DG-automorphism. We then outline a case in which these two constructions are compatible.
	
	\subsubsection*{Skew group $A_\infty$-categories}
	
	This construction was first given in \cite[Definition~5.6]{OZ} (see also \cite[Section 2.4]{AP}).
	
	\noindent Let $G$ be a group and $\Acat$ an $A_\infty$-category. A strict $G$-action on $\Acat$ is a group homomorphism $\rho$ from $G$ to the group of invertible $A_\infty$-endofunctors of $\Acat$. By a small abuse of notation, if $g$ is an element of $G$, we denote the associated functor $g$ instead of $\rho(g)$. 
	
	\begin{definition}
		\label{Ainftyorbit}
		Let $\Acat$ be an $A_\infty$-category and $\rho$ a strict $G$-action on $\Acat$. The skew group $A_\infty$-category $\Acat *_\rho G$ associated to this action is defined as follows.
		\begin{itemize}
			\item Objects of $\Acat *_\rho G$ are the same as those of $\Acat$. If $X$ is an object of $\Acat$, we write $\{X\}$ the corresponding object in $\Acat *_\rho G$.
			\item If $\{X\}$ and $\{Y\}$ are objects, set 
			\[\Hom_{\Acat *_\rho G}(\{X\},\{Y\}):= \bigoplus_{g \in G} \Hom_\Acat(gX,Y)\]
			If $a: gX \to Y$ is a morphism in $\Acat$, we write $a \otimes g$ the corresponding morphism in $\Hom_{\Acat *_\rho G}(\{X\},\{Y\})$.
			\item The $A_\infty$-structure is given on composable homogeneous morphisms by
			\[\mu_{\Acat * G}^n(a_n \otimes g_n, \dots, a_1\otimes g_1):= \mu_\Acat^n(a_n, g_n a_{n-1}, \dots, g_n \dots g_2 a_1) \otimes g_n \dots g_1\]	
		\end{itemize}
	\end{definition}
	
	\begin{remark}
		There is a cohomologically faithful functor $\Acat \hookrightarrow \Acat *_\rho G$ sending a morphism $a: X \to Y$ to $a \otimes 1_G$.
	\end{remark}
	
	\subsubsection*{DG orbit categories}
	
	Let $\Dcat$ be a DG-category and $F: \Dcat \to \Dcat$ an invertible DG-functor. For more details on this construction, see \cite{KelOrb}.
	
	\begin{definition}
		\label{orbcat}
		The DG orbit category $\Dcat/F$ is defined as follows.
		\begin{itemize}
			\item Objects of $\Dcat/F$ are the same as the ones of $\Dcat$;
			\item If $X$ and $Y$ are objects of $\Dcat$ then $\Hom_{\Dcat/F}(X,Y):= \displaystyle\bigoplus_{\ell \in \Z} \Hom_{\Dcat}(X,F^\ell Y)$. The differential is induced by the one in $\Dcat$.
		\end{itemize}
		Note that this is a simplified version of the definition given in \cite{KelOrb}. This is possible because we assume a stronger condition on $F$, namely that it is invertible as a DG-functor.
	\end{definition}
	
	\begin{remark}{\cite[Section 5.1]{KelOrb}}
		\label{DGorb}
		With the setup of Definition~\ref{orbcat}, there is an isomorphism of categories $H^0(\Dcat/F) \simeq H^0(\Dcat)/H^0(F)$ where the second quotient is given by seeing $H^0(\Dcat)$ as a DG-category with zero differential.
	\end{remark}
	
	\begin{proposition}
		\label{sameorbit}
		Let $\Acat$ be a DG-category and $F: \Acat \to \Acat$ an invertible DG-functor. Then $F$ defines a strict $\Z$ action on $\Acat$ (with the $A_\infty$-structure given in Remark \ref{Ainftystr}). There is a isomorphism of $A_\infty$-categories (which is also an isomorphism of DG-categories)
		\[\Fcat: \Acat *_F \Z \overset{\sim}{\to} \Acat / F.\]
	\end{proposition}
	
	\begin{proof}
		As $\Acat *_F \Z$ and $\Acat/F$ have the same objects, define $\Fcat$ to be the identity on objects. Let $X$ and $Y$ be objects of $\Acat$, a morphism $f: \{X\} \to \{Y\}$ in $\Acat *_F \Z$ decomposes as $f = \d\sum_{\ell \in \Z} a_\ell \otimes \sigma^\ell$ with $a_\ell: F^\ell X \to Y$. We then set
		\[\Fcat(f):= \sum_{\ell \in \Z} F^{\ell}(a_{-\ell}) \in \bigoplus_{\ell \in \Z} \Hom_\Acat(X,F^\ell Y)\]
		To show that it is a strict $A_\infty$-functor, it is sufficient to show that it is compatible with $\mu^1_{\Acat * \Z}$ and $\mu^2_{\Acat * \Z}$ as $\mu^d_{\Acat * \Z}=0$ for $d > 2$. Let $\partial$ be the differential in $\Acat$ (from its DG-category structure).
		\begin{align*}
			\Fcat(\mu^1_{\Acat * \Z}(a \otimes \sigma^\ell)) &= \Fcat(\mu^1_\Acat(a) \otimes \sigma^\ell) \\
			&= F^{-\ell}(\mu^1_\Acat(a)) \\
			&= F^{-\ell}((-1)^{|a|}\partial(a)) \\
			&= (-1)^{|a|}\partial(F^{-\ell}(a))\\
			&= (-1)^{|a| + |F^{-\ell}(a)|} \mu^1_{\Acat/F}(\Fcat(a \otimes \sigma^\ell)) \\
			&= \mu^1_{\Acat/F}(\Fcat(a \otimes \sigma^\ell))
		\end{align*}
		
		\noindent Similarly, we have
		
		\begin{align*}
			\Fcat(\mu^2_{\Acat * \Z}(a \otimes \sigma^\ell, b \otimes \sigma^k)) &= \Fcat(\mu^2_\Acat(a,F^\ell(b)) \otimes \sigma^{\ell+k}) \\
			&= F^{-\ell-k}(\mu^2_\Acat(a,F^{\ell}(b))) \\
			&= (-1)^{|F^\ell(b)|} F^{-\ell-k}(a \circ F^{\ell}(b)) \\
			&= (-1)^{|b|} F^{-\ell-k}(a) \circ F^{-k}(b)) \\
			&= (-1)^{|b|+|F^{-k}(b)|}\mu^2_{\Acat/F}(F^{-\ell}(a),F^{-k}(b)) \\
			&= \mu^2_{\Acat/F}(\Fcat(a \otimes \sigma^\ell),\Fcat(b \otimes \sigma^k))
		\end{align*}
		
		\noindent Thus $\Fcat$ is a strict $A_\infty$-functor. 
		
		\noindent Furthermore, we claim that it is an isomorphism. Its inverse is the identity on the objects and sends $a: X \to F^\ell(Y)$ in $\Hom_{\Acat/F}(X,Y)$ to $F^{-\ell}(a) \otimes \sigma^{-\ell}$. It is straightforward to check that this also defines a strict $A_\infty$-functor.
	\end{proof}
	
	\subsection{Triangulated hulls of orbit categories}
	
	In this section, we compare the constructions of triangulated hulls for $A_\infty$-categories and pretriangulated hulls of DG-categories. A DG-category $\Dcat$ is said to be pretriangulated if it is triangulated as an $A_\infty$-category in the sense of Definition~\ref{deftria}. 
	
	\begin{definition}{\cite[Section 5.4]{KelOrb}}
		Let $\Acat$ be a DG-category, $\textup{DGmod} \, \Acat$ be the DG-category of DG-$\Acat$-modules and $\Acat \hookrightarrow \textup{DGmod} \, \Acat$ the Yoneda embedding. We define the pretriangulated hull of $\Acat$, denoted $\pretr(\Acat)$, to be the smallest pretriangulated subcategory of $\textup{DGmod}(\Acat)$ containing $\Acat$. 
		If $\Ccat$ is a category with DG enhancement $\Acat$, we denote $\Ccat_{tr}$ its triangulated hull, defined as 
		\[\Ccat_{tr}:= H^0(\pretr(\Acat))\]
		This construction comes with an additive functor $\Ccat \to \Ccat_{tr}$.
	\end{definition}
	
	\begin{proposition}{\cite[Section 7.4]{Lef}}
		\label{Kenji}
		Let $\Acat$ be a DG-category. There is a quasi-equivalence
		\[y'':\Tw(\Acat) \to \pretr(\Acat)\]
	\end{proposition}
	
	Combining this proposition with Proposition~\ref{sameorbit}, we have the following equivalence.
	
	\begin{corollary}
		\label{pretrequiv}
		Let $\Acat$ be a DG-category, then $\Tw(\Acat)$ is also a DG-category. Let $F: \Tw(\Acat) \to \Tw(\Acat)$ be an invertible DG-functor, it defines a strict $\Z$-action on $\Tw(\Acat)$ in the sense of Definition~\ref{Ainftyorbit} Consider the DG-category $\Tw(\Acat) *_F \Z$. There is an equivalence of triangulated categories
		\[H^0(\pretr(\Tw(\Acat) /F)) \simeq H^0(\Tw(\Tw(\Acat) * \Z)).\]
	\end{corollary}
	
	\section{The one-periodic $A_\infty$-category of twisted complexes}
	
	As group actions on $A_\infty$-categories are written multiplicatively, we consider $\Z$ as the free group with generator $\sigma$.
	
	\begin{definition}
		\label{Ainftyaction}
		Consider the strict $\Z$-action on $\Z \Acat$ induced by the shift functor on graded vector spaces. To be precise, the action of $\sigma$ is given in the following way.
		\begin{itemize}
			\item If $X = \d\bigoplus_{i \in I} X_i[p_i]$ is an object of $\Z \Acat$, define $\sigma X = \d\bigoplus_{i \in I} X_i[p_i +1]$. We denote this object $X[1]$.
			\item If $s^p_q \otimes a: X[p] \to Y[q]$ is a homogeneous morphism in $\Z \Acat$, set 
			\[\sigma (s^p_q \otimes a):= (-1)^{p-q} s^{p+1}_{q+1} \otimes a: X[1] \to Y[1].\] 
			We extend this definition to all morphisms by linearity.
		\end{itemize}
		This strict $\Z$-action extends to a strict $\Z$-action on $\Tw(\Acat)$ in the natural way. This action is the action of the shift on $\Tw(\Acat)$ as defined in \cite[Section 3d]{Sei}.
	\end{definition}
	
	From this $\Z$-action, we can form the $A_\infty$-categories $\Z \Acat *_{[1]} \Z$ and $\Tw(\Acat) *_{[1]} \Z$. From now on, as we only consider the action of the shift, we omit the index $[1]$ and write $\Z\Acat * \Z$ and $\Tw(\Acat) * \Z$. 
	
	In $\Tw(\Acat) * \Z$, all objects are isomorphic to their shifts, making it one-periodic. However, it is not always a triangulated $A_\infty$-category. We thus study its triangulated hull $\Tw(\Tw(\Acat)*\Z)$ instead.
	
	\subsection{$\Tw(\Tw(\Acat) * \Z)$ is quasi-equivalent to $\Tw(\Z \Acat * \Z)$}
	
	\noindent Consider the following strict inclusions of $A_\infty$-categories
	\[\Acat \hookrightarrow \Z \Acat \hookrightarrow \Z \Acat * \Z\]
	
	\noindent Both of these functors are faithful but only the first one is full. Thus they induce a faithful $A_\infty$-functor 
	\[\nu: \Tw(\Acat) \hookrightarrow \Tw(\Z \Acat * \Z)\]
	
	The aim of this section is to show that $\nu$ induces a quasi-equivalence $\tilde{\nu}: \Tw(\Tw(\Acat)* \Z) \to \Tw(\Z \Acat * \Z)$.
	
	\noindent We first give explicit descriptions of $\Tw(\Z \Acat * \Z)$ and of $\nu$.
	
	\vspace{5mm}
	
	\noindent $\bullet$ \underline{The category $\Z(\Z \Acat * \Z)$.}
	
	\noindent Objects of $\Z(\Z \Acat * \Z)$ are of the form $\displaystyle\bigoplus_{i \in I} \{X_i[p_i]\}[q_i]$. The morphism spaces of $\Z(\Z \Acat * \Z)$ are given by
	\[\Hom_{\Z(\Z \Acat * \Z)}(\{X[p]\}[m],\{Y[q]\}[n]) = \Hom_{k}(k[m],k[n]) \otimes \bigoplus_{\ell \in \Z} \Hom_{k}(k[p+\ell],k[q]) \otimes \Hom_\Acat(X,Y)\]
	
	\noindent A morphism in $\Hom_{k}(k[m],k[n]) \otimes \Hom_{k}(k[p+\ell],k[q]) \otimes \Hom_\Acat(X,Y)$ is of the form $s^m_n \otimes s^{p+\ell}_q \otimes a \otimes \sigma^l$ with $a \in \Hom_\Acat(X,Y)$. The degree of this morphism is $m-n+p+\ell-q+|a|$. 
	
	\vspace{5mm}
	
	\noindent $\bullet$ \underline{The category $\Tw(\Z \Acat * \Z)$.}
	
	\noindent Objects of $\Tw(\Z \Acat * \Z)$ are pairs $(X,\delta_X)$ where $X$ is an object of $\Z(\Z \Acat * \Z)$ and $\delta_X$ is a degree 1 endomorphism of $X$ in $\Z(\Z \Acat * \Z)$. Morphisms in $\Tw(\Z \Acat * \Z)$ are the same as the ones in $\Z(\Z \Acat * \Z)$.
	
	\noindent The functor $\nu: \Tw(\Acat) \to \Tw(\Z \Acat * \Z)$ can now be fully described. Let $(X,\delta_X)$ be a twisted complex over $\Acat$ with decompositions $X = \d\bigoplus_{i \in I} X_i[p_i]$ and $\delta_X = \d\sum_{i,j \in I} s^{p_i}_{p_j} \otimes a_{i,j}$. We have
	\[\nu(X,\delta_X) = \left(\bigoplus_{i \in I} \{X_i[0]\}[p_i], \sum_{i,j \in I} s^{p_i}_{p_j} \otimes s^0_0 \otimes a_{i,j} \otimes \sigma^0 \right).\]
	If $a: X \to Y$ is a morphism in $\Acat$ then $\nu(s^m_n \otimes a) = s^m_n \otimes s^0_0 \otimes a \otimes \sigma^0$.
	
	\begin{theorem}
		\label{equione}
		There is a quasi-equivalence $\tilde{\nu}: \Tw(\Tw(\Acat)*\Z) \overset{\sim}{\longrightarrow} \Tw(\Z \Acat * \Z)$
	\end{theorem}
	
	\begin{proof}
		We extend $\nu: \Tw(\Acat) \to \Tw(\Z \Acat * \Z)$ to a cohomologically fully faithful functor $\tilde{\nu}: \Tw(\Acat)* \Z \to \Tw(\Z \Acat * \Z)$ in the following way.
		
		\noindent As $\Tw(\Acat)* \Z$ has the same objects as $\Tw(\Acat)$, define $\tilde{\nu}$ to be the same as $\nu$ on objects. 
		
		\noindent To define $\tilde{\nu}$ on morphisms, recall that
		
		\[\Hom_{\Tw(\Acat) * \Z}((X,\delta_X),(Y,\delta_Y)) = \bigoplus_{\ell \in \Z} \Hom_{\Z \Acat}(X[\ell],Y)\]
		
		If $X = X_i[p_i]$ and $Y = Y_j[q_j]$ with $X_i$ and $Y_j$ in  $\Acat$ then this becomes
		\[\Hom_{\Tw(\Acat) * \Z}((X,\delta_X),(Y,\delta_Y)) = \bigoplus_{\ell \in \Z} \Hom_k (k[p_i+\ell],k[q_j]) \otimes \Hom_\Acat(X_i,Y_j)\]
		Let $s^{p_i+\ell}_{q_j} \otimes a \otimes \sigma^\ell$ be a morphism in $\Hom_{\Tw(\Acat) * \Z}((X,\delta_X),(Y,\delta_Y))$, define
		\[\tilde{\nu}(s^{p_i+\ell}_{q_j} \otimes a \otimes \sigma^\ell) = s^{p_i}_{q_j} \otimes s^\ell_0 \otimes a \otimes \sigma^\ell \in \Hom_{\Z(\Z \Acat * \Z)}(\{X_i(0)\}[p_i], \{Y_j(0)\}[q_j])\]
		Extend this definition to $\Tw(\Acat) * \Z$ by linearity.
		By definition of morphisms in $\Tw(\Z \Acat * \Z)$, $\tilde{\nu}$ is well defined. It remains to show that $\tilde{\nu}$ is an $A_\infty$-functor. 
		
		What follows is the proof of the compatibility of $\tilde{\nu}$ with $\mu_{\Tw(\Acat) * \Z}^2$. This is sufficient in the case where $\Acat$ is an ordinary category.
		
		\noindent Let $\{(X,\delta_X)\}$, $\{(Y,\delta_Y)\}$ and $\{(Z,\delta_Z)\}$ be objects of $\Tw(\Acat) * \Z$. Let $s^{q+\ell_1}_{r} \otimes \alpha \otimes \sigma^{\ell_1}: (Y,\delta_Y) \to (Z,\delta_Z)$ and $s^{p+\ell_2}_{q} \otimes \beta \otimes \sigma^{\ell_2}: (X,\delta_X) \to (Y,\delta_Y)$ be homogeneous morphisms in $\Tw(\Acat) * \Z$ and $\ell:= \ell_1 + \ell_2$. Then
		
		\begin{align*}
			&\tilde{\nu}(\mu^2_{\Tw(\Acat) * \Z}(s^{q+\ell_1}_{r} \otimes \alpha \otimes \sigma^{\ell_1},s^{p+\ell_2}_{q} \otimes \beta \otimes \sigma^{\ell_2})) \\
			&= \tilde{\nu}((-1)^{\ell_1 (p + \ell_2-q)} \mu^2_{\Tw(\Acat)}(s^{q+\ell_1}_{r} \otimes \alpha,s^{p+\ell}_{q+\ell_1} \otimes \beta ) \otimes \sigma^\ell) \\
			& = \tilde{\nu}(\displaystyle \sum_{a,b,c \geq 0} (-1)^{\ell_1 (p+\ell_2-q)} \mu^{2+a+b+c}_{\Z \Acat}(\delta_Z,\dots,\delta_Z, s^{q+\ell_1}_{r} \otimes\alpha, \delta_{Y[\ell_1]}, \dots, \delta_{Y[\ell_1]}, s^{p+\ell}_{q+\ell_1} \otimes \beta,\\
			& \delta_{X[\ell]}, \dots, \delta_{X[\ell]})\otimes \sigma^\ell)
		\end{align*}
		
		The differentials $\delta_X, \delta_Y$ and $\delta_Z$ can be decomposed as sums of morphisms of the form $s^{k'}_k \otimes d$ with $d$ a homogeneous morphism in $\Acat$. Using the fact that $(s^{k'}_k \otimes d) [1] = (-1)^{k'-k} s^{k'+1}_{k+1} \otimes d$ and by linearity, the above computation can be reduced to a sum of terms of the following form (with the added assuption that morphisms considered here are composable).
		
		\begin{align*}
			&\tilde{\nu}((-1)^{\ell_1 (p+\ell_2-q)} \mu^{2+a+b+c}_{\Z \Acat}(s^{r_{c-1}}_{r_c} \otimes d_Z^{c},\dots,s^{r}_{r_1} \otimes d_Z^1, s^{q+l_1}_{r} \otimes \alpha, (-1)^{\ell_1(q_{b-1}-q)}s^{q_{b-1}+\ell_1}_{q+\ell_1} \otimes d_Y^b, \dots, \\
			&(-1)^{\ell_1(q - q_1)}s^{q+\ell_1}_{q_1+\ell_1} \otimes d_Y^1, s^{p+\ell}_{q+\ell_1} \otimes \beta, (-1)^{\ell(p_{a-1} - p)}s^{p_{a-1}+\ell}_{p+\ell} \otimes d_X^a, \dots, (-1)^{\ell(p_0 - p_1)}s^{p_0+\ell}_{p_1+\ell} \otimes d_X^1) \otimes \sigma^\ell) \\
		\end{align*}
		
		This expression can be computed in the following way.
		
		\begin{align*}
			& (-1)^{\ell_1(p-q + \ell_2) + \ell(p_0 - p)} \tilde{\nu}(\mu^{2+a+b+c}_{\Z \Acat}(s^{r_{c-1}}_{r_c} \otimes d_Z^{c},\dots,s^{r}_{r_1} \otimes d_Z^1, s^{q+\ell_1}_{r} \otimes \alpha,s^{q_{b-1}+\ell_1}_{q+\ell_1} \otimes d_Y^b, \dots,\\
			&s^{q+\ell_1}_{q_1 + \ell_1} \otimes d_Y^1, s^{p+\ell}_{q+\ell_1} \otimes \beta,s^{p_{a-1}+\ell}_{p+\ell} \otimes d_X^a, \dots,s^{p_0+\ell}_{p_1+\ell} \otimes d_X^1) \otimes \sigma^\ell) \\
			&= (-1)^{\top} \tilde{\nu} (s^{p_0+\ell}_{r_c} \otimes \mu_\Acat^{2 +a+b+c}(d_Z^c,\dots,d_Z^1,\alpha,d_Y^b,\dots,d_Y^1,\beta, d_X^a, \dots, d_X^1) \otimes \sigma^\ell) \\
			&= (-1)^{\top} s^{p_0}_{r_c} \otimes s^\ell_0 \otimes \mu_\Acat^{2 +a+b+c}(d_Z^c,\dots,d_Z^1,\alpha,d_Y^b,\dots,d_Y^1,\beta, d_X^a, \dots, d_X^1) \otimes \sigma^\ell
		\end{align*}
		
		\noindent where 
		\begin{align*}
			\top &= \ell_1(p-q + \ell_2) + \ell(p_0 - p) + \sum_{i=1}^{c} \| d_Z^i \| (p_0 +\ell - r_{i-1}) + \sum_{i=1}^{b} \| d_Y^i\|(p_0 + \ell -(q_{i-1} + \ell_1))\\
			& + \sum_{i=1}^{a} \| d_X^i \|(p_0 + \ell - (p_{i-1} + \ell)) + \|\alpha \|(p_0 + \ell - (q+\ell_1)) + \|\beta \|(p_0 + \ell - (p+\ell))
		\end{align*} 
		with conventions $r_0 = r$ and $q_0 = q$.
		
		\noindent Conversely, we have
		\begin{align*}
			& \mu_{\Tw(\Z \Acat * \Z)}^2(\tilde{\nu}(s^{q + \ell_1}_{r} \otimes \alpha \otimes \sigma^{\ell_1}),\tilde{\nu}(s^{p + \ell_2}_{q} \otimes \beta \otimes \sigma^{\ell_2})) = \mu_{\Tw(\Z \Acat * \Z)}^2(s^{q}_{r} \otimes s^{\ell_1}_0 \otimes \alpha \otimes \sigma^{\ell_1}, s^{p}_{q} \otimes s^{\ell_2}_0 \otimes \beta \otimes \sigma^{\ell_2}) \\
			&= \sum_{a,b,c \geq 0} \mu_{\Z (\Z \Acat * \Z)}^{2+a+b+c}(\nu(\delta_Z),\dots, \nu(\delta_Z),s^{q}_{r} \otimes s^{\ell_1}_0 \otimes \alpha \otimes \sigma^{\ell_1}, \nu(\delta_Y), \dots, \nu(\delta_Y),s^{p}_{q} \otimes s^{\ell_2}_0 \otimes \beta \otimes \sigma^{\ell_2}, \\
			&\nu(\delta_X), \dots, \nu(\delta_X)) \\
		\end{align*}
		
		\noindent Using the same argument as above, this is a sum of terms of the following form (where the morphisms are composable).
		
		\begin{align*}
			&\mu_{\Z(\Z \Acat * \Z)}^{2+a+b+c}(s^{r_{c-1}}_{r_c} \otimes s^0_0 \otimes d_Z^c \otimes \sigma^0, \dots, s^{r}_{r_1} \otimes s^0_0 \otimes d_Z^1 \otimes \sigma^0,s^{q}_{r} \otimes s^{\ell_1}_0 \otimes \alpha \otimes \sigma^{\ell_1}, s^{q_{b-1}}_{q} \otimes s^0_0 \otimes d_Y^b \otimes \sigma^0, \\
			& \dots, s^{q}_{q_1} \otimes s^0_0 \otimes d_Y^1 \otimes \sigma^0, s^{p}_{q} \otimes s^{\ell_2}_0 \otimes \beta \otimes \sigma^{\ell_2}, s^{p_{a-1}}_{p} \otimes s^0_0 \otimes d_X^a \otimes \sigma^0, \dots, s^{p_0}_{p_1} \otimes s^0_0 \otimes d_X^1 \otimes \sigma^0  ) \\
		\end{align*}
		
		\noindent Applying sign rules from the $A_\infty$ structure on $\Z (\Z \Acat * \Z)$, we have.
		
		\begin{align*}
			&= (-1)^{\perp_1} s^{p_0}_{r_c} \otimes \mu_{\Z \Acat * \Z}^{2+a+b+c}(s^0_0 \otimes d_Z^c \otimes \sigma^0, \dots, s^0_0 \otimes d_Z^1 \otimes \sigma^0,s^{\ell_1}_0 \otimes \alpha \otimes \sigma^{\ell_1},s^0_0 \otimes d_Y^b \otimes \sigma^0, \\
			&\dots, s^0_0 \otimes d_Y^1 \otimes \sigma^0, s^{\ell_2}_0 \otimes \beta \otimes \sigma^{\ell_2}, s^0_0 \otimes d_X^a \otimes \sigma^0, \dots, s^0_0 \otimes d_X^1 \otimes \sigma^0) \\
			&= (-1)^{\perp_1 + \perp_2} s^{p_0}_{r_c} \otimes \mu_{\Z \Acat}^{2+a+b+c}(s^0_0 \otimes d_Z^c, \dots, s^0_0 \otimes d_Z^1,s^{\ell_1}_0 \otimes \alpha,s^{\ell_1}_{\ell_1} \otimes d_Y^b, \dots, s^{\ell_1}_{\ell_1} \otimes d_Y^1, s^{\ell}_{\ell_1} \otimes \beta, \\
			&s^\ell_\ell \otimes d_X^a, \dots, s^\ell_\ell \otimes d_X^1) \otimes \sigma^l \\
			&= (-1)^{\perp_1 + \perp_2 + \perp_3} s^{p_0}_{r_c} \otimes s^\ell_0 \otimes \mu_\Acat^{2+a+b+c}(d_Z^c, \dots, d_Z^1, \alpha, d_Y^b, \dots d_Y^1, \beta, d_X^a, \dots d_X^1) \otimes \sigma^\ell
		\end{align*}
		
		where 
		\begin{align*}\perp_1= &\sum_{i=1}^{c} \| d_Z^i\| (p_0 - r_{i-1}) + \sum_{i=1}^{b} \| d_Y^i\|(p_0 - q_{i-1}) \\
			&+ \sum_{i=1}^{a} \| d_X^i\|(p_0 - p_{i-1}) + (\| \alpha \| + \ell_1) (p_0 - q) + (\| \beta \| + \ell_2)(p_0 - p)) \\
		\perp_2= & \ell_1 \ell_2 \\
		\perp_3= & \ell (\sum_{i=1}^{c} \| d_Z^i \|) + \ell_2(\sum_{i=1}^{b} \| d_Y^i \|) + \ell_2 \|\alpha \|
	\end{align*}
		
		\noindent It is now a straightforward computation to check that $\top = \perp_1 + \perp_2 + \perp_3$. 
		
		\noindent Similarly, we check that $\tilde{\nu}$ commutes with all other multiplications. A detailled computation can be found in Annex A.
		
		\noindent By definition of the morphism spaces of $\Tw(\Acat)* \Z$ and of $\Tw(\Z\Acat * \Z)$, and by construction of $\tilde{\nu}$, it is a homologically fully faithfull functor. Thus it induces a quasi-equivalence $\Tw(\Tw(\Acat)* \Z) \to \Tw(\Z \Acat * \Z)$ by Lemma \ref{Seidel}.
	\end{proof}
	
	With this theorem, we have a new description $H^0(\Tw(\Tw(\Acat)*\Z))$ in terms of twisted complexes over $\Z \Acat * \Z$. 
	
	\subsection{$\Tw(\Tw(\Acat) * \Z)$ is quasi-equivalent to $\Tw(k[t,t^{-1}] \otimes_k \Acat)$}
	
	We now give another description of $\Tw(\Tw(\Acat) * \Z)$. To do so, we first introduce the $A_\infty$-category $k[t,t^{-1}] \otimes_k \Acat$.
	
	\begin{definition}
		Consider $k[t,t^{-1}]$ as a graded algebra with $|t|=1$, the $A_\infty$-category $k[t,t^{-1}] \otimes_k \Acat$ is defined in the following way.
		\begin{itemize}
			\item The objects are those of $\Acat$.
			\item For all objects $X$ and $Y$, $\Hom_{k[t,t^{-1}]\otimes \Acat}(X,Y):= k[t,t^{-1}] \otimes_k \Hom_\Acat(X,Y)$.
			\item For every object $X$, the unit is $1 \otimes \id_X$.
			\item The $A_\infty$-structure is given by
			\[\forall d>0, \; \mu^d_{k[t,t^{-1}] \otimes \Acat}(t^{\alpha_d} \otimes a_d, \dots, t^{\alpha_1}\otimes a_1) = (-1)^{\d\sum_{i<j} \alpha_i \|a_j\|} t^{\alpha_d + \dots + \alpha_1} \otimes \mu_\Acat^d(a_d,\dots,a_1)\]
		\end{itemize}
	\end{definition}
	
	\begin{remark}
		It is a straightforward to check that this is an $A_\infty$-category. If $\Acat$ is a DG-category, we recover the tensor product of DG-categories (see \cite[Section 6.1]{KelDG}).
	\end{remark}
	
	\begin{proposition}
		\label{equitwo}
		There is an $A_\infty$-isomorphism $\Z \Acat * \Z \simeq \Z(k[t,t^{-1}] \otimes_k \Acat)$. This isomorphism induces a quasi-equivalence $\Tw(\Z \Acat * \Z) \overset{\sim}{\longrightarrow} \Tw(k[t,t^{-1}] \otimes_k \Acat)$.
	\end{proposition}
	
	\begin{proof}
		We construct the $A_\infty$-functor 
		\[F: \Z \Acat * \Z \to \Z(k[t,t^{-1}] \otimes_k \Acat)\]
		\noindent in the following way.
		\begin{itemize}
			\item For all objects $X$ in $\Acat$ and integers $p$, set $F(\{X[p]\}):= X[p]$;
			\item If $s^{p+\ell}_q \otimes a \otimes \sigma^\ell: \{X[p]\} \to \{Y[q]\}$ is a morphism in $\Z \Acat * \Z$, set 
			\[F(s^{p+\ell}_q \otimes a \otimes \sigma^\ell) = (-1)^{\frac{\ell(\ell+1)}{2}} s^p_q \otimes t^\ell \otimes a \]
		\end{itemize}
		We now check that it is an $A_\infty$-functor. 
		
		\noindent Let $d>0$ and $s^{p_d+\ell_d}_{p_{d+1}} \otimes a_d \otimes \sigma^{\ell_d},\dots,s^{p_1+\ell_1}_{p_2} \otimes a_1 \otimes \sigma^{\ell_1}$ be composable morphisms. Let $\ell:= \d\sum_{i=1}^d \ell_i$, we have:
		
		\begin{align*}
			& F(\mu^d_{\Z\Acat * \Z}(s^{p_d+\ell_d}_{p_{d+1}} \otimes a_d \otimes \sigma^{\ell_d},\dots,s^{p_1+\ell_1}_{p_2} \otimes a_1 \otimes \sigma^{\ell_1})) \\
			&= F((-1)^{\d\sum_{i<j}(p_i + \ell_i - p_{i+1}) \ell_j} \mu^d_{\Z \Acat}(s^{p_d+\ell_d}_{p_{d+1}} \otimes a_d,\dots,s^{p_1+\ell}_{p_2 + \ell-\ell_1} \otimes a_1) \otimes \sigma^\ell) \\
			&= F((-1)^{\d\sum_{i<j}(p_i + \ell_i - p_{i+1})(\ell_j + \|a_j\|)} s^{p_1+\ell}_{p_{d+1}} \otimes \mu^d_\Acat(a_d,\dots,a_1) \otimes \sigma^\ell) \\
			&= (-1)^{\d\sum_{i<j}(p_i + \ell_i - p_{i+1})(\ell_j + \|a_j\|) + \frac{\ell(\ell+1)}{2}} s^{p_1}_{p_{d+1}} \otimes t^\ell \otimes \mu^d_\Acat(a_d,\dots,a_1)
		\end{align*}
		
		Conversely, we have:
		
		\begin{align*}
			& \mu^d_{\Z(k[t,t^{-1}] \otimes_k \Acat)}(F(s^{p_d+\ell_d}_{p_{d+1}} \otimes a_d \otimes \sigma^{\ell_d}),\dots,F(s^{p_1+\ell_1}_{p_2} \otimes a_1 \otimes \sigma^{\ell_1})) \\
			&= (-1)^{\d\sum_{i=1}^{d} \frac{\ell_i(\ell_i+1)}{2}} \mu^d_{\Z(k[t,t^{-1}] \otimes_k \Acat)}(s^{p_d}_{p_{d+1}} \otimes t^{\ell_d} \otimes a_d, \dots, s^{p_1}_{p_2} \otimes t^{\ell_1} \otimes a_1) \\
			&= (-1)^{\d\sum_{i=1}^{d} \frac{\ell_i(\ell_i+1)}{2} + \sum_{i<j} (p_i - p_{i+1})(\ell_j + \|a_j\|)} s^{p_1}_{p_{d+1}} \otimes \mu^d_{k[t,t^{-1}] \otimes_k \Acat}(t^{\ell_d} \otimes a_d, \dots, t^{\ell_1} \otimes a_1) \\
			&= (-1)^{\d\sum_{i=1}^{d} \frac{\ell_i(\ell_i+1)}{2} + \sum_{i<j} (p_i - p_{i+1})(\ell_j + \|a_j\|) + \ell_i \|a_j\|}s^{p_1}_{p_{d+1}} \otimes t^\ell \otimes \mu^d_\Acat(a_d,\dots,a_1)
		\end{align*}
		As the signs coincide, we have shown that $F$ is indeed an $A_\infty$-functor. It is straightforward to check that $F$ is an isomorphism.
		 
		\noindent It follows that $F$ induces a cohomologically fully faithful functor $\Z \Acat * \Z \to \Tw(k[t,t^{-1}] \otimes_k \Acat)$. As $k[t,t^{-1}] \otimes_k \Acat$ is in the image of this functor and is a generator of $\Tw(k[t,t^{-1}] \otimes_k \Acat)$, by Lemma \ref{Seidel}, $F$ induces a quasi-equivalence $\Tw(\Z \Acat * \Z) \overset{\sim}{\to} \Tw(k[t,t^{-1}] \otimes_k \Acat)$.
	\end{proof}
	
	\begin{remark}
		In \cite{Chr}, the $A_\infty$-category $\Tw(k[t,t^{-1}] \otimes_k \Acat)$ is the prototype for the one-periodic Fukaya category of a compact surface with boundary and stops. In this case, $\Acat$ is a category given by a dissection on the surface with the construction of \cite{HKK}. In particular, as $\Acat$ is a category, we have $\Tw(k[t,t^{-1}] \otimes_k \Acat) \simeq \Dcat^b(k[t,t^{-1}] \otimes_k \Acat)$.
	\end{remark}
	
	\subsection{The $A_\infty$-subcategory $\Scat$}
	
	\noindent To further simplify the description of $\Tw(\Tw(\Acat) * \Z)$ given in the previous section, we give a full $A_\infty$-subcategory $\Scat$ of $\Tw(k[t,t^{-1}] \otimes_k \Acat)$ such that the inclusion functor $\Scat \hookrightarrow \Tw(k[t,t^{-1}] \otimes_k \Acat)$ is a quasi-equivalence.
	
	\begin{definition}
		\label{defScat}
		Let $\Scat$ be the full $A_\infty$-subcategory of $\Tw(k[t,t^{-1}] \otimes_k \Acat)$ of objects of the form $(\d\bigoplus_{i \in I} X_i[0],\delta)$. Denote $\iota: \Scat \to \Tw(k[t,t^{-1}] \otimes_k \Acat)$ the inclusion functor (it is an $A_\infty$-functor for the $A_\infty$-structure on $\Scat$ induced by the one on $\Tw(k[t,t^{-1}] \otimes_k \Acat)$).
	\end{definition}
	
	\begin{proposition}
		\label{retract}
		The functor $Z^0(\iota)$ has a retraction (left inverse) and thus is an equivalence of categories. Furthermore, $\iota$ is a quasi-equivalence.
	\end{proposition}
	
	\begin{proof}
		Let $X = (\d\sum_{i \in I} X_i[p_i],\delta_X:= \d\sum_{i,j \in I} \d\sum_{\ell \in \Z} s^{p_i}_{p_j} \otimes t^\ell \otimes a^\ell_{i,j})$ be an object in $\Tw(k[t,t^{-1}] \otimes_k \Acat)$. First we give an isomorphism between $X$ and an object $rX$ of $\Scat$ defined as follows
		\[rX = \left(\bigoplus_{i \in I} X_i[0], r\delta_X = \sum_{i,j \in I} \sum_{\ell \in \Z} (-1)^{(p_i+1)\|a^\ell_{i,j}\| -p_j(p_i-p_j)-\ell} s^0_0 \otimes t^{- \|a^\ell_{i,j}\|} \otimes a^\ell_{i,j} \right)\] 
		The fact that $r\delta_X$ is strictly lower triangular follows from the fact that it the case for $\delta_X$. 
		
		\noindent The isomorphism $\varphi_X: X \to rX$ is defined by
		\[\varphi_X = \sum_{i \in I} s^{p_i}_0 \otimes t^{-p_i} \otimes \id_{X_i}\]
		The next step is to show that $\mu^1_{\Tw(k[t,t^{-1}] \otimes_k \Acat)}(\varphi_X) = 0$. Using linearity and the fact that $\Acat$ is strictly unital, this computation is equivalent to showing that for all $i,j$ in $I$ and all integers $\ell$, we have 
		\begin{align*}
			&\mu^2_{\Z(k[t,t^{-1}] \otimes_k \Acat)}(s^{p_j}_0 \otimes t^{-p_j} \otimes \id_{X_j},s^{p_i}_{p_j} \otimes t^\ell \otimes a^\ell_{i,j}) \\
			&+ \mu^2_{\Z(k[t,t^{-1}] \otimes_k \Acat)}( (-1)^{(p_i+1)\|a^\ell_{i,j}\| -p_j(p_i-p_j)-\ell}s^0_0 \otimes t^{- \|a^\ell_{i,j}\|} \otimes a^\ell_{i,j},s^{p_i}_0 \otimes t^{-p_i} \otimes \id_{X_i}) = 0
		\end{align*}
		Computing the first term of the sum, we have
		
		\begin{align*}
			& \mu^2_{\Z(k[t,t^{-1}] \otimes_k \Acat)}(s^{p_j}_0 \otimes t^{-p_j} \otimes \id_{X_j},s^{p_i}_{p_j} \otimes t^\ell \otimes a^\ell_{i,j}) \\
			&= (-1)^{-p_j(p_i-p_j)} s^{p_i}_0 \otimes \mu^2_{k[t,t^{-1}] \otimes_k \Acat}(t^{-p_j} \otimes \id_{X_j}, t^\ell \otimes a^\ell_{i,j}) \\
			&= (-1)^{-p_j(p_i-p_j)-\ell} s^{p_i}_0 \otimes t^{\ell-p_j} \otimes \mu^2_\Acat(\id_{X_j},a^\ell_{i,j}) \\
			&= (-1)^{-p_j(p_i-p_j)-\ell + |a^\ell_{i,j}|} s^{p_i}_0 \otimes t^{\ell-p_j} \otimes a^\ell_{i,j}
		\end{align*}
		
		\noindent Computing the second term of the sum we have
		
		\begin{align*}
			& \mu^2_{\Z(k[t,t^{-1}] \otimes_k \Acat)}((-1)^{(p_i+1)\|a^\ell_{i,j}\| -p_j(p_i-p_j)-\ell}s^0_0 \otimes t^{- \|a^\ell_{i,j}\|} \otimes a^\ell_{i,j},s^{p_i}_0 \otimes t^{-p_i} \otimes \id_{X_i}) \\
			&= (-1)^{(p_i+1)\|a^\ell_{i,j}\| -p_j(p_i-p_j)-\ell}s^{p_i}_0 \otimes \mu^2_{k[t,t^{-1}] \otimes_k \Acat}(t^{- \|a^\ell_{i,j}\|} \otimes a^\ell_{i,j},t^{-p_i} \otimes \id_{X_i}) \\
			&= (-1)^{\|a^\ell_{i,j}\|  -p_j(p_i-p_j)-\ell} s^{p_i}_0 \otimes t^{-\|a^\ell_{i,j}\| - p_i} \otimes a^\ell_{i,j}
		\end{align*}
		As $p_i - p_j +\ell + \|a^\ell_{i,j}\| = 0$, we have the wanted relation, and $\varphi_X$ is an morphism in $Z^0(\Tw(k[t,t^{-1}] \otimes_k \Acat))$. It is an isomorphism with inverse $\varphi_X^{-1}:= \d\sum_{i \in I} s^0_{p_i} \otimes t^{p_i} \otimes \id_{X_i}$.
		
		Now define the retraction $r$ of $Z^0(\iota)$ as follows:
		\begin{itemize}
			\item if $X$ is an object of $Z^0(\Tw(k[t,t^{-1}] \otimes_k \Acat))$ then $rX$ is defined as above;
			\item if $f:X \to Y$ is a morphism in $Z^0(\Tw(k[t,t^{-1}] \otimes_k \Acat))$ then $r(f):= \varphi_Y \circ f\circ\varphi_X^{-1}$
		\end{itemize}
		By construction, $r:Z^0(\Tw(k[t,t^{-1}] \otimes_k \Acat)) \to Z^0(\Scat)$ is a functor. As $\varphi_X = \id_X$ for $X$ in $\Scat$, $r$ is a retraction of $Z^0(\iota)$. 
		
		\noindent As $Z^0(\iota)$ is an equivalence of categories, so is $H^0(\iota)$.
	\end{proof}
	
	\begin{corollary}
		\label{totalequi}
		There exists an equivalence of triangulated categories between $H^0(\Tw(\Tw(\Acat)*\Z))$ and $H^0(\Scat)$.
	\end{corollary}
	
	\begin{proof}
		This is the composition of quasi-equivalences obtained in Theorem~\ref{equione} and Proposition~\ref{equitwo}. 
	\end{proof}

	\noindent We now give a full description of $\Scat$ and $H^0(\Scat)$.
	
	\begin{itemize}
		\item Objects of $\Scat$ are pairs $(\d\bigoplus_{i \in I} X_i,\delta_X:= \d\sum_{i,j \in I} t^{- \|a_{i,j}\|} \otimes a_{i,j})$ where $(X_i)_{i \in I}$ are objects of $\Acat$ and $(a_{i,j}: X_i \to X_j)_{i,j \in I}$ are morphisms in $\Acat$.
		\item Morphisms in $\Scat$ between $X = (\d\bigoplus_{i \in I} X_i, \delta_X)$ and $Y = (\d\bigoplus_{j \in J} Y_j, \delta_Y)$ are of the form 
		\[\d\sum_{\substack{i \in I \\ j \in J}} \sum_{\ell \in \Z}  t^{\ell} \otimes a^\ell_{i,j}\]
		\item The $A_\infty$-structure is given by the one on $\Tw(k[t,t^{-1}] \otimes_k \Acat)$.
	\end{itemize} 
	
	\noindent We can now describe $H^0(\Scat)$
	
	\begin{itemize}
		\item The objects are the ones of $\Scat$.
		\item Morphisms between $X = (\d\bigoplus_{i \in I} X_i, \delta_X)$ and $Y = (\d\bigoplus_{j \in J} Y_j, \delta_Y)$ are of the form 
		\[f = \d\sum_{i \in I} \sum_{j \in J} t^{-|a_{i,j}|} \otimes a_{i,j}\]
		and must verify $\mu^1_{\Tw(k[t,t^{-1}] \otimes_k \Acat)}(f) = 0$.
		\item Homotopies are given by morphisms in the image of $\mu^1_{\Tw(k[t,t^{-1}] \otimes_k \Acat)}$.
	\end{itemize}
	
	\begin{remark}
		
	\label{suspH0}
	
	The category $H^0(\Scat)$ inherits a triangulated structure from $H^0(\Tw(k[t,t^{-1}] \otimes_k \Acat))$. In particular it has a suspension functor given by 
	\[\Sigma X = r(X[1])\]
	
	\noindent More explicitely, if $X = (\d\bigoplus_{i \in I} X_i,\delta_X= \d\sum_{i,j \in I} t^{- \|a_{i,j}\|} \otimes a_{i,j})$ is an object of $H^0(\Scat)$, its suspension is
	\[\Sigma X = (\d\bigoplus_{i \in I} X_i,\delta_X= \d\sum_{i,j \in I} (-1)^{\|a_{i,j}\|}t^{- \|a_{i,j}\|} \otimes a_{i,j})\]
	
	\end{remark}
	
	\noindent In the next section we will give another description when $\Acat$ is given by a graded category.
	
	\section{The case of graded categories and Cohen-Macaulay modules}
	
	\label{sectionCM}
	
	Let $(Q,I)$ be a graded bound quiver. Consider $\Acat$ the $k$-linear $\Z$-graded category obtained from the pair $(Q,I)$ and equip it with the $A_\infty$-structure given in Remark~\ref{Ainftystr}. We assume that the algebra $A = kQ/I$ is finite dimensional and of finite global dimension. Let $\mod \, A$ be the category of finitely generated right $A$-modules. Every object of $\Acat$ corresponds to a vertex of $Q$ and thus to a unique indecomposable projective summand of $A$. From now on, objects of $\Acat$ are considered to be graded projective $A$-modules with this canonical grading and morphisms in $\Acat$ are considered to be graded morphisms between graded projective $A$-modules.
	
	\noindent With this setup, Corollary~\ref{pretrequiv} induces the following equivalence of triangulated categories.
	
	\begin{proposition}
		There is an equivalence of triangulated categories
		\[(\per(A_{DG})/[1])_{tr} \simeq H^0(\Tw(\Tw(\Acat)*\Z))\]
		where $A_{DG}$ is the DG-algebra $A$ with zero differential and $\per(A_{DG})$ is the triangulated category of perfect DG-$A$-modules.
	\end{proposition}
	
	\begin{proof}
		With the setup above, we have $H^0(\textup{DGmod} \, \Acat) \simeq \per(A_{DG})$. Thus we have $H^0(\Tw(\Acat) * \Z) \simeq \per(A_{DG})/[1]$ by virtue of remark \ref{DGorb} and Proposition~\ref{sameorbit}. Using Proposition~\ref{Kenji}, there is a quasi-equivalence
		\[\Tw(\Tw(\Acat)* \Z) \overset{\sim}{\longrightarrow} \pretr(\Tw(\Acat) * \Z)\]
		Thus we have the wanted equivalence of triangulated categories. 
	\end{proof}
	
	Combining it with the quasi-equivalence between $\Tw(\Tw(\Acat) * \Z)$ and $\Scat$ constructed in the previous section, the triangulated categories $(\per(A_{DG})/[1])_{tr}$ and $H^0(\Scat)$ are equivalent. 
	
	\subsection{Cohen-Macaulay modules and differential modules}
	
	\noindent The aim of this section is to give a description of $H^0(\Scat)$ in this setup in terms of maximal Cohen-Macaulay modules.
	
	\begin{definition}
		\label{Altimes}
		Let $A^\ltimes:= A \otimes_k k[\varepsilon]/\langle \varepsilon^2 \rangle$ be a graded $k$-algebra where $\varepsilon$ has degree 1. We consider $A^\ltimes$ as an ungraded algebra. Note that $A^\ltimes$ is not obtained as the tensor product of $A$ and of $k[\varepsilon]/\langle \varepsilon^2\rangle$ as ungraded algebras. Indeed, multiplication in $A^\ltimes$ still follows the Koszul sign rule:
		\[(a\otimes \varepsilon^\ell).(b \otimes \varepsilon^m) = (-1)^{\ell |b|} ab \otimes \varepsilon^{\ell+m}\]
		
		\noindent If $A$ is concentrated in degree zero, the algebra $A^\ltimes$ is the one defined in \cite{Win}.
		
		\noindent Let $\mod \, A^\ltimes$ be the category of finitely generated $A^\ltimes$-modules. We say that an $A^\ltimes$-module $M$ is maximal Cohen-Macaulay if 
		\[\forall \, i >0, \; \Ext^i_{A^\ltimes}(M,A^\ltimes) = 0\]
		We denote $\CM(A^\ltimes)$ the full subcategory of $\mod \, A^\ltimes$ consisting of maximal Cohen-Macaulay modules.
	\end{definition}
	
	\noindent In our case, as $A$ has finite global dimension, the algebra $A^\ltimes$ is Gorenstein (see \cite[Proposition~2.2]{AR}). Thus the category $\CM(A^\ltimes)$ is a Frobenius exact subcategory of $\mod \, A^\ltimes$ (see \cite{AR,Buc}). Projective-injective objects in this category are projective $A^\ltimes$-modules. Furthermore, in our context maximal Cohen-Macaulay modules over $A^\ltimes$ can be characterized in the following way.
		
	\begin{proposition}{\cite[Proposition~1.2.3]{Win}}
		\label{CMproj}
		Let $M$ be an (ungraded) $A^\ltimes$-module. Then $M$ is maximal Cohen-Macaulay if and only if its image by the restriction functor $\mod \, A^\ltimes \to \mod \, A$ is a projective $A$-module.
	\end{proposition}
	
	\noindent In the case where $A$ is concentrated in degree zero, a description of $\CM(A^\ltimes)$ using differential modules is given in \cite{Win}. We now extend this construction to the graded case with the following construction.
	
	\begin{definition}
		\label{difmod}
		Let $A$ be a $\Z$-graded algebra.
		We define the category of differential $A$-modules as follows.
		\begin{itemize}
			\item Objects are pairs of the form $(M,d_M)$ where $M$ is an (ungraded) $A$-module and $d_M: M \to M$ is a $k$-linear map such that $d_M^2 = 0$ and
			\[\forall m \in M, \; \forall a \in A, \; d_M(m.a) = (-1)^{|a|} d_M(m).a\]
			\item A morphism $f: (M,d_M) \to (N,d_N)$ is the data of a morphism $f:M \to N$ in $\mod \, A$ such that $f \circ d_M = d_N \circ f$ in $\mod \, k$.
		\end{itemize}
		Set $\Pcat(A)$ to be the full subcategory $\Dif(A)$ consisting of objects $(P,d_P)$ where $P$ is a projective $A$-module.
	\end{definition}
	
	\begin{remark}
		If $M$ is a DG-module over $A$, it can be made into a differential module over $A$ by forgetting the $\Z$-grading. This defines a faithful functor $\textup{DGmod} \, A \longrightarrow \Dif(A)$. However, this functor is not dense as not all differential modules can be endowed with a grading making them a DG-module.
	\end{remark}
	
	\begin{proposition}
		\label{equivDif}
		Consider the functor $\Gamma: \mod \, A^\ltimes \to \Dif(A)$ defined as follows
		\begin{itemize}
			\item If $M$ is an $A^\ltimes$-module then $\Gamma(M) = (M_{|A},d_M)$ where $M_{|A}$ is the restriction of $M$ to $\mod \, A$ and $d_M$ is given by the action of $1\otimes \varepsilon$ on $M$. 
			\item If $f: M \to N$ is a morphism in $\mod \, A^\ltimes$, define $\Gamma(f): (M_{|A},d_M) \to (N_{|A},d_N)$ to be the restriction $f_{|A}$ of $f$ to $\mod \, A$.
		\end{itemize}
		This functor is an equivalence of categories that restricts to an equivalence 
		\[\CM(A^\ltimes) \overset{\sim}{\to} \Pcat(A).\]
	\end{proposition}
	
	\begin{proof}
		We first check that for $M$ in $\mod \, A^\ltimes$, $\Gamma(M)$ is indeed a differential $A$-module. It is sufficient to check that for all $m$ in $M$, $a$ in $A$ we have
		\[(m.1\otimes \varepsilon).a = (-1)^{|a|} (m.a).1 \otimes \varepsilon\] 
		This is true from the Koszul sign rule on $A^\ltimes$. 
		Furthermore, if $f: M \to N$ is a morphism in $\mod \, A^\ltimes$, then $f_{|A} \circ d_M = d_N \circ f_{|A}$ follows from the definition of morphisms in $\mod \, A^\ltimes$. 	
		
		\noindent We now claim that $\Gamma$ is an equivalence of categories, with inverse functor defined as follows.
		\begin{itemize}
			\item A differential module $(M,d_M)$ is sent to the $A^\ltimes$-module with underlying vector space $M$ (seen as a $k$-vector space) and $A^\ltimes$ action given by
			\[\forall m \in M, \; \forall a \in A, \; m.(a \otimes \varepsilon^\ell) = (m.(a \otimes 1)).(1 \otimes \varepsilon^\ell) = d_M^\ell(m.a)\]
			\item If $f:(M,d_M) \to (N,d_N)$ is a morphism in $\Dif(A)$ then the underlying morphism of $A$-modules $f: M \to N$ defines a morphism of $A^\ltimes$ modules thanks to the relation $f \circ d_M = d_N \circ f$.
		\end{itemize}
		It is straightforward to see that this defines a functor that is inverse to $\Gamma$.
		
		The fact that $\Gamma$ restricts to an equivalence between $\CM(A^\ltimes)$ and $\Pcat(A)$ follows from Proposition~\ref{CMproj}.
	\end{proof}
	
	From this equivalence, we can endow $\Dif(A)$ with the structure of an abelian category. With this structure, $\Pcat(A)$ is a Frobenius exact subcategory of $\Dif(A)$.
	
	\begin{definition}
		Let $P$ be a projective $A$-module. Consider the $k$-linear map $\varepsilon_P: P \to P$ defined on homogeneous elements by
		\[\varepsilon_P(m) = (-1)^{|m|}m\]
		Here, the grading on $P$ is the canonical grading used to see objects of $\Acat$ as graded projective $A$-modules.
	\end{definition}
	
	\begin{corollary}
		The projective-injectives objects of $\Pcat(A)$ that are isomorphic to objects of the form $(P\oplus P, \begin{pmatrix}
			0 & 0 \\
			\varepsilon_P & 0
		\end{pmatrix}).$
	\end{corollary}
	
	\begin{proof}
		The projective-injective objects of $\Pcat(A)$ are the images by $\Gamma$ of projective $A^\ltimes$-modules. The projective $A^\ltimes$ modules are of the form $P \otimes k[\varepsilon]/\langle \varepsilon^2 \rangle$ with $P$ an (ungraded) projective $A$-module (i.e. a summand of $A^{\oplus n}$ for some $n>0$, without its grading). In this case we have 
		\[\Gamma(P \otimes k[\varepsilon]/\langle \varepsilon^2 \rangle) = (P\oplus P, \begin{pmatrix}
			0 & 0 \\
			\varepsilon_P & 0
		\end{pmatrix})\]
	\end{proof}
	
	\begin{remark}
		\label{exactstr}
		The exact structure on $\Pcat(A)$ is induced by the one on $\mod \, A$. Indeed, a sequence $0 \to (X,\delta_X) \overset{f}{\to} (Y,\delta_Y) \overset{g}{\to} (Z,\delta_Z) \to 0$ is exact in $\Pcat(A)$ if and only if $0 \to X \overset{f}{\to} Y \overset{g}{\to} Z \to 0$ is exact in $\mod \, A$.
	\end{remark}
	
	\begin{remark}
		\label{susp}
		The suspension in the triangulated category $\underline{\CM(A^\ltimes)}$ is given by the syzygy functor $\Omega$. This enables us to compute the suspension in $\underline{\Pcat(A)}$ in the following way. Let $(P,d_P = \sum_{i \in \Z} d_P^i)$ be an object of $\underline{\Pcat(A)}$ where $d_P^i$ is homogeneous of degree $i$. We have the following short exact sequence.
		
		\[0 \to (P,d_P[1]) \overset{\begin{pmatrix}
				\overline{d_P} \\
				-\id_P
		\end{pmatrix}}{\longrightarrow} (P \oplus P, \begin{pmatrix}
			0 & 0 \\
			\varepsilon_P & 0
		\end{pmatrix}) \overset{\begin{pmatrix}
			\id_P & \overline{d_P}
			\end{pmatrix}}{\longrightarrow} (P, d_P) \to0 \]
		where $d_P[1] = \d\sum_{i \in \Z} (-1)^{1 + |d_P^i|} d_P^i$.
	\end{remark}
	
	\subsection{The category $H^0(\Scat)$ as a category of differential modules}
	We now claim that the stable category $\underline{\Pcat(A)}$ gives a new description of $H^0(\Scat)$. 
	
	\begin{proposition}
		\label{exactequi}
		There is an equivalence of exact categories $Z^0(\Scat) \simeq \Pcat(A)$.
	\end{proposition}
	
	\begin{proof}
		We construct an equivalence $\Gcat: Z^0(\Scat) \to \Pcat(A)$ as follows:
		\begin{itemize}
			\item If $X = (\d\bigoplus_{i \in I} X_i,\delta_X = \sum_{i,j \in I} t^{- \|a_{i,j}\|} \otimes a_{i,j})$ is an object of $Z^0(\Scat)$, define $\Gcat(X)$ to be the pair $(\d\bigoplus_{i \in I} X_i,d_X = \sum_{i,j \in I} (-1)^{\s(\|a_{i,j}\|)} \overline{a_{i,j}})$ where $\s(x) = \frac{x(x+1)}{2}$ and $\overline{a_{i,j}}: X_i \to X_j$ is defined by $\overline{a_{i,j}}(m) = (-1)^{|m|} a_{i,j}(m)$. Remark that $\s(x+y) = \s(x) + \s(y) + xy$ and that if $a_{i,j}$ and $a_{j,k}$ are homogeneous then $\overline{a_{j,k}} \circ \overline{a_{i,j}} = (-1)^{|a_{i,j}|} a_{j,k} \circ a_{i,j}$ (this defines an involution).
			\item If $f = \d\sum_{i \in I} \sum_{j \in J} t^{-|f_{i,j}|} \otimes f_{i,j}$ is a morphism between $X$ and $Y$ (using the notations above), set $\Gcat(f) = \d\sum_{i \in I} \sum_{j \in J} (-1)^{\s(|f_{i,j}|)}f_{i,j}$.
		\end{itemize}	
		We now check that $\Gcat$ is a functor.
		
		First we check that if $X$ an object of $Z^0(\Scat)$ with notations above, $\Gcat(X)$ is indeed an object of $\Pcat(A)$. To do so, it is sufficient to check that $d_X^2 = 0$. The generalized Maurer-Cartan equations gives us $\mu^2_{k[t,t^{-1}] \otimes_k \Acat}(\delta_X,\delta_X)= 0$. Thus we have
		\[\sum_{i,j,k \in I} \mu^2_{k[t,t^{-1}] \otimes_k \Acat}(t^{- \|a_{j,k}\|} \otimes a_{j,k},t^{- \|a_{i,j}\|} \otimes a_{i,j}) = 0\]
		i.e.
		\[\sum_{i,j,k \in I} (-1)^{-\|a_{i,j}\|.\|a_{j,k}\| + |a_{i,j}|} t^{-\|a_{i,j}\|-\|a_{j,k}\|} \otimes a_{j,k} \circ a_{i,j} = 0\] 
		By grouping terms in the sum, we have
		
		\[\sum_{\ell \in \Z} t^{-\ell} \otimes \sum_{\substack{i,j,k \in I \\ \|a_{i,j}\| + \|a_{j,k}\| = \ell}} (-1)^{-\|a_{i,j}\|.\|a_{j,k}\| + |a_{i,j}|} a_{j,k} \circ a_{i,j} = 0\] 
		Thus for all integers $l$, we have
		\[\sum_{\substack{i,j,k \in I \\ \|a_{i,j}\| + \|a_{j,k}\| = \ell}} (-1)^{-\|a_{i,j}\|.\|a_{j,k}\| + |a_{i,j}|} a_{j,k} \circ a_{i,j} = 0\]
		
		\noindent However, we have 
		\begin{align*}
			d_X^2 &= \sum_{i,j,k \in I} (-1)^{\s(\|a_{i,j}\|) + \s(\|a_{j,k}\|) + |a_{i,j}|} a_{j,k} \circ a_{i,j} \\
			&= \sum_{\ell \in \Z} (-1)^{\s(\ell)} \sum_{\substack{i,j,k \in I \\ \|a_{i,j}\| + \|a_{j,k}\| = \ell}} (-1)^{-\|a_{i,j}\|.\|a_{j,k}\| + |a_{i,j}|}  a_{j,k} \circ a_{i,j} \\
			&= 0
		\end{align*}
		\noindent Thus $\Gcat(X)$ is indeed an object of $\Pcat(A)$.

		\noindent We now check that $\Gcat$ respects composition. Consider morphisms of $Z^{0}(\Scat)$
		\[t^{-|f|} \otimes f: (Y = \bigoplus_{i \in J} Y_i,\delta_Y = \sum_{i,j \in J} t^{- \|b_{i,j}\|} \otimes b_{i,j}) \to (Z = \bigoplus_{i \in K} Z_i,\delta_Z = \sum_{i,j \in K} t^{- \|c_{i,j}\|} \otimes c_{i,j})\]
		 and
		\[t^{-|g|} \otimes g: (X = \bigoplus_{i \in I} X_i,\delta_X = \sum_{i,j \in I} t^{- \|a_{i,j}\|} \otimes a_{i,j}) \to (Y = \bigoplus_{i \in J} Y_i,\delta_Y = \sum_{i,j \in J} t^{- \|b_{i,j}\|} \otimes b_{i,j})\]
		We have
		
		\begin{align*}
			\Gcat(t^{-|f|} \otimes f \circ t^{-|g|} \otimes g) &= \Gcat(\mu^2_{k[t,t^{-1}]\otimes_k \Acat}(t^{-|f|} \otimes f,t^{-|g|} \otimes g)) \\
			&= \Gcat((-1)^{-|g|.\|f\|}t^{-|f| - |g|} \otimes \mu^2_\Acat(f,g)) \\
			&= (-1)^{-|g|.\|f\| + |g|} \Gcat(t^{-|f| - |g|} \otimes f \circ g) \\
			&= (-1)^{|g|.|f| + \s(-|f|-|g|)} f \circ g \\
			&= (-1)^{|g|.|f|+ \s(-|f|-|g|) - \s(-|f|) - \s(-|g|)} \Gcat(t^{-|f|} \otimes f) \circ \Gcat(t^{-|g|} \otimes g) \\
			&= \Gcat(t^{-|f|} \otimes f) \circ \Gcat(t^{-|g|} \otimes g)
		\end{align*} 
		
		\noindent Consider a morphism $t^{-|f|} \otimes f: (X, \delta_X) \to (Y,\delta_Y)$ where $(X,\delta_X)$ and $(Y,\delta_Y)$ are as above and $f: X_k \to Y_\ell$ is a homogeneous morphism in $\Acat$. Remark also that if $a$ and $b$ are homogeneous morphisms of $\Acat$ then $a \circ \overline{b} = \overline{a \circ b}$ and $\overline{a} \circ b = (-1)^{|b|} \overline{a \circ b}$. We now have
		\begin{align*}
			\Gcat(f) &\circ d_X - d_Y \circ \Gcat(f) = \sum_{\substack{i \in I \\ j \in J}} (-1)^{\s(|f|) + \s(\|a_{i,k}\|)} f \circ \overline{a_{i,k}} - (-1)^{\s(|b_{\ell,j}|) + \s(|f|)} \overline{b_{\ell,j}} \circ f \\
			&= \sum_{\substack{i \in I \\ j \in J}} (-1)^{\s(|f|) + \s(\|a_{i,k}\|)} \overline{f \circ a_{i,k}} - (-1)^{\s(\|b_{\ell,j}\|) + \s(|f|) + |f|} \overline{b_{\ell,j} \circ f} \\
			&= \sum_{r \in \Z} \left( \sum_{\substack{i \in I \\ j \in J \\ \|a_{i,k}\|+|f| = r}} (-1)^{\s(|f|) + \s(\|a_{i,k}\|)} \overline{f \circ a_{i,k}} - \sum_{\substack{i \in I \\ j \in J \\ \|b_{\ell,j}\|+|f| = r}} (-1)^{\s(\|b_{\ell,j}\|) + \s(|f|) + |f|} \overline{b_{\ell,j} \circ f} \right)
		\end{align*}
		However, $\mu^1_{\Scat}(t^{-|f|} \otimes f) = 0$. Therefore
		
		\begin{align*}
			\mu^1_{\Scat}(t^{-|f|} \otimes f) &= \mu^2_{k[t,t^{-1}] \otimes_k \Acat}(t^{-|f|} \otimes f,\delta_X) + \mu^2_{k[t,t^{-1}] \otimes_k \Acat}(\delta_Y,t^{-|f|} \otimes f) \\
			&= \sum_{\substack{i \in I \\ j \in J}} \mu^2_{k[t,t^{-1}] \otimes_k \Acat}(t^{-|f|} \otimes f,t^{- \|a_{i,k}\|} \otimes a_{i,k}) +  \mu^2_{k[t,t^{-1}] \otimes_k \Acat}(t^{- \|b_{\ell,j}\|} \otimes b_{\ell,j},t^{-|f|} \otimes f) \\
			&= \sum_{\substack{i \in I \\ j \in J}} (-1)^{\|a_{i,k}\|.\|f\| + |a_{i,k}|} t^{-|f|-\|a_{i,k}\|} \otimes f \circ a_{i,k} +(-1)^{|f|.\|b_{\ell,j}\| + |f|} t^{-\|b_{\ell,j}\|-|f|} \otimes b_{\ell,j} \circ f \\
			&= \sum_{r \in \Z} t^{-r} \otimes \left(\sum_{\substack{i \in I \\ j \in J \\ \|a_{i,k}\|+|f| = r}} (-1)^{\|a_{i,k}\|.|f| + 1} f \circ a_{i,k} + \sum_{\substack{i \in I \\ j \in J \\ \|b_{\ell,j}\|+|f| = r}} (-1)^{|f|.\|b_{\ell,j}\| + |f|}  b_{\ell,j} \circ f \right) \\
			&= \sum_{r \in \Z} (-1)^{\s(r)} t^{-r} \otimes \left( \sum_{\substack{i \in I \\ j \in J \\ \|a_{i,k}\|+|f| = r}} (-1)^{\s(\|a_{i,k}\|)+\s(|f|) + 1} f \circ a_{i,k} \right) \\
			& + (-1)^{\s(r)} t^{-r} \otimes \left(\sum_{\substack{i \in I \\ j \in J \\ \|b_{\ell,j}\|+|f| = r}} (-1)^{\s(|f|)+\s(\|b_{\ell,j}\|) + |f|}  b_{\ell,j} \circ f \right)  \\
			&= 0
		\end{align*}
		Thus for all integers $r$, we have 
		\[\left(\sum_{\substack{i \in I \\ j \in J \\ \|a_{i,k}\|+|f| = r}} (-1)^{\s(\|a_{i,k}\|)+\s(|f|) + 1} f \circ a_{i,k} + \sum_{\substack{i \in I \\ j \in J \\ \|b_{\ell,j}\|+|f| = r}} (-1)^{\s(|f|)+\s(\|b_{\ell,j}\|) + |f|}  b_{\ell,j} \circ f \right) = 0\]
		This gives us the wanted relation $\Gcat(f) \circ d_X - d_Y \circ \Gcat(f) = 0$. 
		
		We have now shown that $\Gcat$ is indeed a functor. From its definition on objects, we deduce that it is essentially surjective: in fact if $X = \left(\d\bigoplus_{i \in I} P_i, d_X = \sum_{i,j \in I} a_{i,j}\right)$ is an object of $\Pcat(A)$, we can construct an essential preimage $\left(\d\bigoplus_{i \in I} P_i, \delta_X = \sum_{i,j \in I} (-1)^{\s(\|\overline{a_{i,j}}\|)}t^{-\|\overline{a_{i,j}}\|} \otimes \overline{a_{i,j}}\right)$. Similarly, we can construct an inverse map on morphism spaces to show that $\Gcat$ is fully faithful.
		Using Remark \ref{exactstr} it is straightforward to see that $\Gcat$ is exact.
	\end{proof}
	
	\begin{theorem}
		\label{Frobequi}
		There is an equivalence of triangulated categories $H^0(\Scat) \simeq \underline{\Pcat(A)}$.
	\end{theorem}
	
	\begin{proof}
		We have an exact equivalence $\Gcat$ (from Proposition~\ref{exactequi}) between two Frobenius exact categories. Therefore it induces a equivalence of triangulated categories between the respective stable categories. As the stable category of $Z^0(\Scat)$ is $H^0(\Scat)$, this induced equivalence is the wanted one.
	\end{proof}
	
	\begin{remark}
		By comparing Remark \ref{suspH0} and Remark \ref{susp}, we see that signs for the suspension in $\underline{\Pcat(A)}$ and $H^0(\Scat)$ do coincide.
	\end{remark}
	
	\begin{corollary}
		\label{main}
		There is an equivalence of categories $(\per(A)/[1])_{tr} \simeq \underline{\Pcat(A)}$
	\end{corollary}
	
	\begin{proof}
		This equivalence is obtained by composing the equivalence induced by the quasi-equivalences of Corollary~\ref{totalequi} and the triangulated equivalence of Theorem~\ref{Frobequi}.
	\end{proof}
	
	We now give an explicit description of the functor $\per(A) \to \underline{\Pcat(A)}$. This functor is defined by applying $H^0$ to the following $A_\infty$-functors 
	
	\begin{tikzpicture}
		\node (1) at (0,0) {$\Tw(\Acat)$};
		\node (2) at (3,0) {$\Tw(\Tw(\Acat)* \Z)$};
		\node (3) at (7,0) {$\Tw(\Z \Acat * \Z)$};
		\node (4) at (12,0) {$\Tw(k[t,t^{-1}] \otimes_k \Acat)$};
		\node (5) at (16,0) {$\Scat$};
		
		\draw[-stealth] (1) to (2);
		\draw[-stealth] (2) to node[midway, above] {$\tilde{\nu}$} (3);
		\draw[-stealth] (3) to node[midway, above] {$F$} (4);
		\draw[-stealth] (5) to node[midway, above] {$\iota$} (4);
		\draw[-stealth, dotted, bend left=30] (4) to node[above, midway] {$r$} (5);
	\end{tikzpicture}
	and then applying the equivalence $H^0(\Scat) \simeq \underline{\Pcat(A)}$ of Theorem~\ref{Frobequi}.
	
	\noindent $\bullet $ Let $X=(P,d_P)$ be an object of $\per(A)$ where $P$ is a graded projective $A$-module and $d_P$ is a differential. Decompose $P$ as $P\simeq \d\bigoplus_{i \in I} P_i[p_i]$ where each $P_i$ is an indecomposable summand of $A$. Decompose $d_P$ as $d_P= \d\sum_{i,j \in I} d_{i,j}$ where each $d_{i,j}: P_i[p_i] \to P_j[p_j]$ is a degree one morphism. It follows that $p_i-p_j + \|d_{i,j}\| = 0$.
	
	\noindent Seen as an object of $\Tw(\Acat)$, $X$ is the twisted complex
	\[(\bigoplus_{i \in I} P_i[p_i], \sum_{i,j \in I} s^{p_i}_{p_j} \otimes d_{i,j})\]
	
	\noindent Its image in $\Tw(\Z \Acat * \Z)$ by $\tilde{\nu}$ is the twisted complex
	
	\[(\bigoplus_{i \in I} \{P_i(0)\}[p_i], \sum_{i,j \in I} s^{p_i}_{p_j} \otimes s^0_0 \otimes d_{i,j} \otimes \sigma^0)\]
	
	\noindent In turn, its image in $\Tw(k[t,t^{-1}] \otimes_k \Acat)$ by $F$ is the twisted complex
	
	\[(\bigoplus_{i \in I} P_i[p_i], \sum_{i,j \in I} s^{p_i}_{p_j} \otimes t^0 \otimes d_{i,j})\]
	
	\noindent Applying the retraction $r$ of Proposition~\ref{retract}, its image in $Z^0(\Scat)$ (and in $H^0(\Scat)$) is
	
	\[(\bigoplus_{i \in I} P_i, \sum_{i,j \in I} (-1)^{(p_i+1)\|d_{i,j}\| - p_j(p_i-p_j)} t^0 \otimes d_{i,j} )\]
	
	\noindent We can now compute its image in $\Pcat(A)$ (and in $\underline{\Pcat(A)}$):
	
	\[(\bigoplus_{i \in I} P_i, \sum_{i,j \in I} (-1)^{(p_i+1)\|d_{i,j}\| - p_j(p_i-p_j) + \s(\|d_{i,j}\|)} \overline{d_{i,j}})\]
	
	\noindent In particular, if $A$ is in degree zero, the image of $X$ in $\underline{\Pcat(A)}$ is
	
	\[(\bigoplus_{i \in I} P_i, \sum_{i,j \in I} (-1)^{(p_i+1)} d_{i,j})\]
	
	\noindent $\bullet$ Let $f: (P,d_P) \to (Q,d_Q)$ be a morphism in $\per(A)$ with $(P,d_P)$ as above and $(Q,d_Q) = (\d\sum_{j \in J} Q_j(q_j), \sum_{j,k \in J} \partial_{j,k})$. Decompose $f$ as $f= \d\sum_{\substack{i \in I \\ j\in J}} f_{i,j}$ with $f_{i,j}: P_i(p_i) \to Q_j(q_j)$ a degree zero map (thus $p_i - q_j + |f_{i,j}|=0$). The image of $f$ in $\underline{\Pcat(A)}$ is 
	\[\sum_{\substack{i \in I \\ j\in J}} (-1)^{q_j(q_j-p_i) + \s(|f_{i,j}|)} f_{i,j}\]
	
	\noindent As before, if $A$ is in degree zero (if $A$ is ungraded for example), the image of $f$ becomes
	
	\[\sum_{\substack{i \in I \\ j\in J}} f_{i,j}\]
	
	\begin{remark}
		The functor $\per(A) \to \underline{\Pcat(A)}$ obtained here in the ungraded case coincides with the one described in \cite{Win}.
	\end{remark}
	
	\section{Annex A: Complete proof of Theorem~\ref{equione}}
	
	In this annex, we give the full proof of Theorem~\ref{equione}. It remains to check that $\tilde{\nu}$ defines an $A_\infty$-functor in all generality, i.e. coincides with $\mu^d$ for all $d >0$ and not just $d=2$.
	
	\begin{proof}
	
	Let $d >0$ be an integer and $(s^{p_k + l_k}_{p_{k+1}} \otimes a_k \otimes \sigma^{l_k})_{1 \leq k \leq d}$ be a family of composable morphisms in $\Tw(\Acat) * \Z$ where $s^{p_k + l_k}_{p_{k+1}} \otimes a_k \otimes \sigma^{l_k}: (X_{k-1},\delta_{k-1}) \to (X_k,\delta_k)$. As each $\delta_k$ is an endomorphism of $X_k$ in $\Z \Acat$, we can decompose them as $\delta_k = \displaystyle\sum_{j \in J_k} s^{q_{k,j}}_{r_{k,j}} \otimes a_{k,j}$. Let $l:= \d\sum_{i=1}^d l_i$, we now have
	
	\begin{align*}
		& \tilde{\nu}(\mu_{\Tw(\Acat) * \Z}^d(s^{p_d + l_d}_{p_{d+1}} \otimes a_d \otimes \sigma^{l_d},\dots,s^{p_1 + l_1}_{p_{2}} \otimes a_1 \otimes \sigma^{l_1})) \\
		&= \tilde{\nu}((-1)^{\Delta_1} \mu_{\Tw(\Acat)}^d(s^{p_d + l_d}_{p_{d+1}} \otimes a_d,s^{p_{d-1} + l_d + l_{d-1}}_{p_d + l_d} \otimes a_{d-1},\dots,s^{p_1 + l}_{p_2 + l - l_1} \otimes a_1) \otimes \sigma^{l}) \\
		&= \tilde{\nu}((-1)^{\Delta_1} \sum_{k_0,\dots,k_d} \mu_{\Z \Acat}^{d + k_0 + \dots + k_d}(\delta_d,\dots,\delta_d,s^{p_d + l_d}_{p_{d+1}} \otimes a_d,\delta_{d-1}[l_d],\dots,s^{p_1+l}_{p_2 + l - l_1} \otimes a_1, \delta_1[l],\dots,\delta_1[l])\otimes \sigma^l) \\
		&= \tilde{\nu}((-1)^{\Delta_1} \sum_{k_0,\dots,k_d} \sum_{(j(x))_{1 \leq x \leq k} \in \d\prod_{i=0}^d J_i^{k_i}} (-1)^{\Delta_2(j)} \mu_{\Z \Acat}^{d + k_0 + \dots + k_d}(s^{q_{d,j(\sum_{i=0}^{d} k_i)}}_{r_{d,j(\sum_{i=0}^{d} k_i)}}\otimes a_{d,j(\sum_{i=0}^{d} k_i)},\dots,s^{q_{d,j(1 +\sum_{i=0}^{d-1} k_i)}}_{r_{d,j(1+\sum_{i=0}^{d-1} k_i)}} \\
		&\otimes a_{d,j(1+\sum_{i=0}^{d-1} k_i)},s^{p_d + l_d}_{p_{d+1}} \otimes a_d,\dots,s^{p_1+l}_{p_2 + l - l_1} \otimes a_1, s^{q_{1,j(k_0)}+l}_{r_{1,j(k_0)}+l} \otimes a_{1,j(k_0)},\dots,s^{q_{1,j(1)}+l}_{r_{1,j(1)}+l}\otimes a_{1,j(1)})\otimes \sigma^l) \\
		&= \tilde{\nu}((-1)^{\Delta_1} \sum_{k_0,\dots,k_d} \sum_{(j(x))_{1 \leq x \leq k} \in \d\prod_{i=0}^d J_i^{k_i}} (-1)^{\Delta_2(j) + \Delta_3(j)} s^{q_{1,j(1)}+l}_{r_{d,j(\sum_{i=0}^{d}k_i)}} \otimes \mu_\Acat^{d + k_0 + \dots + k_d}(a_{d,j(\sum_{i=0}^{d} k_i)},\dots, \\
		&a_{d,j(1+\sum_{i=0}^{d-1} k_i)},a_d,\dots, a_1, a_{1,j(k_0)},\dots,a_{1,j(1)})\otimes \sigma^l) \\
		&= (-1)^{\Delta_1} \sum_{k_0,\dots,k_d} \sum_{(j(x))_{1 \leq x \leq k} \in \d\prod_{i=0}^d J_i^{k_i}} (-1)^{\Delta_2(j) + \Delta_3(j)} s^{q_{1,j(1)}}_{r_{d,j(\sum_{i=0}^{d}k_i)}} \otimes s^l_0 \otimes \mu_\Acat^{d + k_0 + \dots + k_d}(a_{d,j(\sum_{i=0}^{d} k_i)},\dots, \\
		&a_{d,j(1+\sum_{i=0}^{d-1} k_i)},a_d,\dots, a_1, a_{1,j(k_0)},\dots,a_{1,j(1)})\otimes \sigma^l 
	\end{align*} 
	where 
	\[k = \sum_{i=0}^{d} k_i\]
	\[\ind(j(x)) = i \;\textup{if}\; j(x) \in J_i\]
	\[\Delta_1 = \sum_{i=1}^{d-1} (\sum_{k=i+1}^{d} l_k) (p_i + l_i - p_{i+1}) \]
	\[\Delta_2(j) = \sum_{x=1}^{k} \left( \sum_{m=\ind(j(x))+1}^{d} l_m\right) (q_{\ind(j(x)),j(x)}-r_{\ind(j(x)),j(x)})\]
	\begin{align*}\Delta_3(j) =& \sum_{x<y} (q_{\ind(j(x)),j(x)} - r_{\ind(j(x)),j(x)}) \|a_{\ind(j(y)),j(y)}\| \\
		&+ \sum_{x=1}^{k} (q_{\ind(j(x)),j(x)} - r_{\ind(j(x)),j(x)}) \left(\sum_{m=\ind(j(x))+1}^{d} \|a_m\| \right) \\
		& + \|a_{\ind(j(x)),j(x)}\|\left(\sum_{m=1}^{\ind(j(x))} (p_m +l_m-p_{m+1})\right) + \sum_{i<j} (p_i+l_i-p_{i+1})\|a_j\| 
	\end{align*}
	
	Conversely, we have
	
	\begin{align*}
		& \mu_{\Tw(\Z \Acat * \Z)}^d(\tilde{\nu}(s^{p_d + l_d}_{p_{d+1}} \otimes a_d \otimes \sigma^{l_d}),\dots,\tilde{\nu}(s^{p_1 + l_1}_{p_2} \otimes a_1 \otimes \sigma^{l_1})) \\
		&= \mu_{\Tw(\Z \Acat * \Z)}^d (s^{p_d}_{p_{d+1}} \otimes s^{l_d}_0 \otimes a_d \otimes \sigma^{l_d},\dots,s^{p_1}_{p_2} \otimes s^{l_1}_0 \otimes a_1 \otimes \sigma^{l_1}) \\
		&= \sum_{k_0,\dots,k_d} \sum_{(j(x))_{1 \leq x \leq k} \in \d\prod_{i=0}^d J_i^{k_i}} \mu_{\Z(\Z \Acat * \Z)}^{d+k_0+\dots+k_d}(s^{q_{d,j(\sum_{i=0}^{d} k_i)}}_{r_{d,j(\sum_{i=0}^{d} k_i)}} \otimes s^0_0 \otimes a_{d,j(\sum_{i=0}^{d}k_i)}\otimes \sigma^0,\dots, \\
		&s^{q_{d,j(1+\sum_{i=0}^{d-1} k_i)}}_{r_{d,j(1+\sum_{i=0}^{d-1} k_i)}} \otimes s^0_0 \otimes a_{d,j(1+\sum_{i=0}^{d-1}k_i)}\otimes \sigma^0,s^{p_d}_{p_{d+1}} \otimes s^{l_d}_0 \otimes a_d \otimes \sigma^{l_d},\dots,s^{p_1}_{p_2} \otimes s^{l_1}_0 \otimes a_1 \otimes \sigma^{l_1}, \\
		&s^{q_{1,j(k_0)}}_{r_{1,j(k_0)}} \otimes s^0_0 \otimes a_{1,j(k_0)} \otimes \sigma^0,\dots,s^{q_{1,j(1)}}_{r_{1,j(1)}} \otimes s^0_0 \otimes a_{1,j(1)} \otimes \sigma^0) \\
		&= \sum_{k_0,\dots,k_d} \sum_{(j(x))_{1 \leq x \leq k} \in \d\prod_{i=0}^d J_i^{k_i}} (-1)^{\Sigma_1(j)} s^{q_{1,j(1)}}_{r_{d,j(\sum_{i=0}^{d} k_i)}} \otimes \mu_{\Z \Acat * \Z}^{d+k_0+\dots+k_d}(s^0_0 \otimes a_{d,j(\sum_{i=0}^{d}k_i)}\otimes \sigma^0,\dots, \\
		&s^0_0 \otimes a_{d,j(1+\sum_{i=0}^{d-1}k_i)}\otimes \sigma^0,s^{l_d}_0 \otimes a_d \otimes \sigma^{l_d},\dots,s^{l_1}_0 \otimes a_1 \otimes \sigma^{l_1},s^0_0 \otimes a_{1,j(k_0)} \otimes \sigma^0,\dots,s^0_0 \otimes a_{1,j(1)} \otimes \sigma^0) \\
		&= \sum_{k_0,\dots,k_d} \sum_{(j(x))_{1 \leq x \leq k} \in \d\prod_{i=0}^d J_i^{k_i}} (-1)^{\Sigma_1(j) + \Sigma_2(j)} s^{q_{1,j(1)}}_{r_{d,j(\sum_{i=0}^{d} k_i)}} \otimes \mu_{\Z \Acat}^{d+k_0+\dots+k_d}(s^0_0 \otimes a_{d,j(\sum_{i=0}^{d}k_i)},\dots, \\
		&s^0_0 \otimes a_{d,j(1+\sum_{i=0}^{d-1}k_i)},s^{l_d}_0 \otimes a_d,\dots,s^{l}_{l-l_1} \otimes a_1,s^l_l \otimes a_{1,j(k_0)},\dots,s^l_l \otimes a_{1,j(1)})\otimes \sigma^l \\
		&= \sum_{k_0,\dots,k_d} \sum_{(j(x))_{1 \leq x \leq k} \in \d\prod_{i=0}^d J_i^{k_i}} (-1)^{\Sigma_1(j) + \Sigma_2(j) + \Sigma_3(j)} s^{q_{1,j(1)}}_{r_{d,j(\sum_{i=0}^{d} k_i)}} \otimes s^l_0 \otimes \mu_{\Acat}^{d+k_0+\dots+k_d}(a_{d,j(\sum_{i=0}^{d}k_i)},\dots, \\
		&a_{d,j(1+\sum_{i=0}^{d-1}k_i)},a_d,\dots,a_1,a_{1,j(k_0)},\dots,a_{1,j(1)})\otimes \sigma^l \\
	\end{align*}
	where
	\begin{align*} \Sigma_1(j) =& \sum_{x<y} (q_{\ind(j(x)),j(x)} - r_{\ind(j(x)),j(x)}) \|a_{\ind(j(y)),j(y)}\| \\
		&+ \sum_{x=1}^{k} (q_{\ind(j(x)),j(x)} - r_{\ind(j(x)),j(x)}) \left(\sum_{m=\ind(j(x))+1}^{d} l_m + \|a_m\|\right) \\
		& + \sum_{x=1}^{k} \left(\sum_{i=1}^{\ind(j(x))}(p_i - p_{i+1})\right) \| a_{\ind(j(x)),j(x)} \| + \sum_{i<j} (p_i-p_{i+1})(l_j + \|a_j\|) \end{align*}
	\[\Sigma_2(j) = \sum_{i=1}^{d-1} l_i \left(\sum_{j=i+1}^{d} l_j\right) \]
	\[\Sigma_3(j) = \sum_{x=1}^{k} \|a_{\ind(j(x)),j(x)}\| \left(\sum_{m=1}^{\ind(j(x))} l_m \right) + \sum_{i=1}^{d} \|a_i\| \left(\sum_{m=1}^{i-1} l_m\right)\]
	
	A straightforward computation shows that for all $(j(x))_{1\leq x \leq k}$ in $\prod_{i=0}^d J_i^k$, $\Delta_1 + \Delta_2(j) + \Delta_3(j) = \Sigma_1(j)+\Sigma_2(j) + \Sigma_3(j)$. Thus $\tilde{\nu}$ is a strict $A_\infty$-functor.
	
	\end{proof}
	
	\bibliographystyle{alpha}
	\bibliography{biblio}
\end{document}